\documentclass[twoside,leqno,10pt, A4]{amsart}
\usepackage{amsfonts}
\usepackage{amsmath}
\usepackage{amscd}
\usepackage{amssymb}
\usepackage{amsthm}
\usepackage{amsrefs}
\usepackage{latexsym}
\usepackage{mathrsfs}
\usepackage{bbm}
\usepackage{amscd}
\usepackage{amssymb}
\usepackage{amsthm}
\usepackage{amsrefs}
\usepackage{latexsym}
\usepackage{mathrsfs}
\usepackage{bbm}
\usepackage{enumerate}
\usepackage{graphicx}
\allowdisplaybreaks
\usepackage{color}
\begin{document}

\newtheorem{theorem}[subsection]{Theorem}
\newtheorem{proposition}[subsection]{Proposition}
\newtheorem{lemma}[subsection]{Lemma}
\newtheorem{corollary}[subsection]{Corollary}
\newtheorem{conjecture}[subsection]{Conjecture}
\newtheorem{prop}[subsection]{Proposition}
\newtheorem{defin}[subsection]{Definition}

\numberwithin{equation}{section}
\newcommand{\mr}{\ensuremath{\mathbb R}}
\newcommand{\mc}{\ensuremath{\mathbb C}}
\newcommand{\N}{\mathbb{N}}
\newcommand{\dif}{\mathrm{d}}
\newcommand{\intz}{\mathbb{Z}}
\newcommand{\ratq}{\mathbb{Q}}
\newcommand{\natn}{\mathbb{N}}
\newcommand{\comc}{\mathbb{C}}
\newcommand{\rear}{\mathbb{R}}
\newcommand{\prip}{\mathbb{P}}
\newcommand{\uph}{\mathbb{H}}
\newcommand{\fief}{\mathbb{F}}
\newcommand{\majorarc}{\mathfrak{M}}
\newcommand{\minorarc}{\mathfrak{m}}
\newcommand{\sings}{\mathfrak{S}}
\newcommand{\fA}{\ensuremath{\mathfrak A}}
\newcommand{\mn}{\ensuremath{\mathbb N}}
\newcommand{\mq}{\ensuremath{\mathbb Q}}
\newcommand{\half}{\tfrac{1}{2}}
\newcommand{\f}{f\times \chi}
\newcommand{\summ}{\mathop{{\sum}^{\star}}}
\newcommand{\chiq}{\chi \bmod q}
\newcommand{\chidb}{\chi \bmod db}
\newcommand{\chid}{\chi \bmod d}
\newcommand{\sym}{\text{sym}^2}
\newcommand{\hhalf}{\tfrac{1}{2}}
\newcommand{\sumstar}{\sideset{}{^*}\sum}
\newcommand{\sumprime}{\sideset{}{'}\sum}
\newcommand{\sumprimeprime}{\sideset{}{''}\sum}
\newcommand{\sumflat}{\sideset{}{^\flat}\sum}
\newcommand{\shortmod}{\ensuremath{\negthickspace \negthickspace \negthickspace \pmod}}
\newcommand{\V}{V\left(\frac{nm}{q^2}\right)}
\newcommand{\sumi}{\mathop{{\sum}^{\dagger}}}
\newcommand{\mz}{\ensuremath{\mathbb Z}}
\newcommand{\leg}[2]{\left(\frac{#1}{#2}\right)}
\newcommand{\muK}{\mu_{\omega}}
\newcommand{\thalf}{\tfrac12}
\newcommand{\lp}{\left(}
\newcommand{\rp}{\right)}
\newcommand{\Lam}{\Lambda_{[i]}}
\newcommand{\lam}{\lambda}
\newcommand{\af}{\mathfrak{a}}
\newcommand{\sw}{S_{[i]}(X,Y;\Phi,\Psi)}
\newcommand{\lz}{\left(}
\newcommand{\pz}{\right)}
\newcommand{\bfrac}[2]{\lz\frac{#1}{#2}\pz}
\newcommand{\odd}{\mathrm{\ primary}}
\newcommand{\even}{\text{ even}}
\newcommand{\res}{\mathrm{Res}}
\newcommand{\sumn}{\sumstar_{(c,1+i)=1}  w\left( \frac {N(c)}X \right)}
\newcommand{\lab}{\left|}
\newcommand{\rab}{\right|}
\newcommand{\Go}{\Gamma_{o}}
\newcommand{\Ge}{\Gamma_{e}}
\newcommand{\M}{\widehat}
\def\su#1{\sum_{\substack{#1}}}
\newcommand{\Echar}{\mathbb{E}^{\text{char}}}
\newcommand{\E}{\mathbb{E}}
\newcommand{\p}{\mathbb{P}}

\theoremstyle{plain}
\newtheorem{conj}{Conjecture}
\newtheorem{remark}[subsection]{Remark}

\newcommand{\pfrac}[2]{\left(\frac{#1}{#2}\right)}
\newcommand{\pmfrac}[2]{\left(\mfrac{#1}{#2}\right)}
\newcommand{\ptfrac}[2]{\left(\tfrac{#1}{#2}\right)}
\newcommand{\pMatrix}[4]{\left(\begin{matrix}#1 & #2 \\ #3 & #4\end{matrix}\right)}
\newcommand{\ppMatrix}[4]{\left(\!\pMatrix{#1}{#2}{#3}{#4}\!\right)}
\renewcommand{\pmatrix}[4]{\left(\begin{smallmatrix}#1 & #2 \\ #3 & #4\end{smallmatrix}\right)}
\def\en{{\mathbf{\,e}}_n}

\newcommand{\ppmod}[1]{\hspace{-0.15cm}\pmod{#1}}
\newcommand{\ccom}[1]{{\color{red}{Chantal: #1}} }
\newcommand{\acom}[1]{{\color{blue}{Alia: #1}} }
\newcommand{\alexcom}[1]{{\color{green}{Alex: #1}} }
\newcommand{\hcom}[1]{{\color{brown}{Hua: #1}} }

\makeatletter
\def\widebreve{\mathpalette\wide@breve}
\def\wide@breve#1#2{\sbox\z@{$#1#2$}%
     \mathop{\vbox{\m@th\ialign{##\crcr
\kern0.08em\brevefill#1{0.8\wd\z@}\crcr\noalign{\nointerlineskip}%
                    $\hss#1#2\hss$\crcr}}}\limits}
\def\brevefill#1#2{$\m@th\sbox\tw@{$#1($}%
  \hss\resizebox{#2}{\wd\tw@}{\rotatebox[origin=c]{90}{\upshape(}}\hss$}
\makeatletter

\title[High moments of random multiplicative functions with Fourier coefficients of modular forms]{High moments of random multiplicative functions twisted by Fourier coefficients of modular forms}

\author[P. Gao]{Peng Gao}
\address{School of Mathematical Sciences, Beihang University, Beijing 100191, China}
\email{penggao@buaa.edu.cn}

\author[L. Zhao]{Liangyi Zhao}
\address{School of Mathematics and Statistics, University of New South Wales, Sydney NSW 2052, Australia}
\email{l.zhao@unsw.edu.au}

\begin{abstract}
  Let $\lambda(n)$ denote the Fourier coefficients of a fixed modular form and $h(n)$ a Steinhaus or Rademacher random multiplicative function.  In this paper, we determine, under the generalized Riemann hypothesis, the order of magnitude of $\E|\sum_{n \leq x} h(n)\lambda(n)|^{2q}$ up to factors of size $e^{O(q^2)}$, for all real $x, q$ with $1 \leq q \leq c\log x/\log\log x $ and $c>0$ a small constant. 
\end{abstract}

\maketitle



\section{Introduction}\label{sec1}

  In a sequence of work \cites{Harper20, harper2020moments, Harper23}, A. J. Harper determined the order of magnitude of moments of sums of Steinhaus and Rademacher random multiplicative functions and extended his methods to arithmetic functions to show remarkably that the typical sizes of partial sums of character and zeta sums exhibit better than square-root cancellation.  We reserve the letter $p$ for a prime number throughout the paper.  Recall from \cite{Harper20} that a Steinhaus random multiplicative function $h(n)$ for all $n \in \N$ is defined to be $h(n) = \prod_{p^{a} \| n} h(p)^{a}$, where $(h(p))_{p \; \text{prime}}$ is a sequence of independent random variables, each distributed uniformly on the complex unit circle. Also, a Rademacher random multiplicative function $h(n)$ is defined to be supported on square-free integers $n$ only such that $h(n) = \prod_{p |n} h(p)$, where $(h(p))_{p \; \text{prime}}$ is a sequence of independent random variables taking values $\pm 1$ with probability $1/2$ each. \newline

Let $f$ be a fixed holomorphic Hecke eigenform of weight $\kappa \equiv 0 \pmod 4$ for the full modular group $SL_2 (\mathbb{Z})$, whose Fourier expansion of $f$ at infinity is given by
\[
f(z) = \sum_{n=1}^{\infty} \lambda (n) n^{(\kappa -1)/2} e(nz), \quad \mbox{where} \quad e(z) = \exp (2 \pi i z).
\]
In \cite{G&Zhao26-04}, the authors determined the order of magnitude of the low moments of sums of the form $\sum_{n \leq x} h(n)\lambda(n)$ where $h$ is a Steinhaus or a Rademacher random multiplicative function. The aim of this paper is to continue our work to determine all high moments of the above sum under the generalized Riemann hypothesis (GRH).  Let $\mathbb{E}$ denote the expectation throughout the paper. Our results are as follows.
\begin{theorem}
\label{upperboundsfixedmodmeanS}
Let $h$ be a Steinhaus random multiplicative function.  Then we have, under the truth of GRH and the notation as above, a small absolute constant $c > 0$ such that,
\begin{align*}
 \E |\sum_{n \leq x} h(n)\lambda(n)|^{2q} = \exp \left( -q^2\log q - q^2 \log\log(2q) + O(q^2) \right) x^{q} \log^{(q-1)^2}x,
\end{align*}
uniformly for all large $x$ and $1 \leq q \leq \frac{c\log x}{\log\log x}$.
\end{theorem}

\begin{theorem}
\label{upperboundsfixedmodmeanR}
Let $h$ be a Rademacher random multiplicative function and $q_0 = (1 + \sqrt{5})/2$.  Under the truth of GRH and  the notation as above. There exists a small absolute constant $c > 0$ such that, 
\begin{align*}
 \E |\sum_{n \leq x} & h(n)\lambda(n)|^{2q} \\
 & = \exp \left( -2q^2\log q - 2q^2\log\log(2q) + O(q^2) \right) \left(1 + \min\left\{\log\log x , \frac{1}{|q - q_0 |}\right\}\right) x^{q} \log^{\max\{(q-1)^2, q(2q-3)\}}x,
\end{align*}
uniformly for all large $x$ and real $1 \leq q \leq \frac{c\log x}{\log\log x}$.
\end{theorem}

The proofs of the above theorems follow closely the treatments in \cite{harper2020moments} with the incorporation of the arithmetic properties of $\lambda(n)$.  We shall prove the validity of the bounds 
\begin{align}
\label{mainestimationupper}
\begin{split} 
 \E |  \sum_{n \leq x}  h(n)\lambda(n)|^{2q}  \leq 
\begin{cases}
S(q,x) \exp \left( O(q^2) \right) , \quad  & \text{if $h$ is Steinhaus}, \\ \\
R(q,x) \exp \left( O(q^2) \right) , \quad  & \text{if $h$ is  Rademacher},
\end{cases}
\end{split}
\end{align}
  and
\begin{align}
\label{mainestimationlower}
\begin{split} 
\E |\sum_{n \leq x} h(n)\lambda(n)|^{2q} \geq 
\begin{cases}
S(q,x) \exp \left( O(q^2) \right) , \quad  & \text{if $h$ is Steinhaus}, \\ \\
R(q,x) \exp \left( O(q^2) \right) , \quad  & \text{if $h$ is  Rademacher},
\end{cases}
\end{split}
\end{align}
where
\[ S(q,x) = \exp \left( -q^{2}\log q - q^{2}\log\log(2q) \right) x^{q} \log^{(q-1)^2}x \]
and
\[ R(q,x) = \exp \left( -2q^{2}\log q - 2q^{2}\log\log(2q) \right) \left(1 + \min\left\{\log\log x , \frac{1}{|q - q_0 |}\right\}\right)  x^{q} \log^{\max\{(q-1)^2, q(2q-3)\}}x . \]
Theorems \ref{upperboundsfixedmodmeanS} and \ref{upperboundsfixedmodmeanR} follow from these bounds in \eqref{mainestimationupper} and \eqref{mainestimationlower} readily. \newline
 
We remark here that our assumption on GRH is only indispensable at two places in the paper.  The powerful conjecture is necessary to show the second term in \eqref{uppersmoothingdisplay} is negligible and to estimate from below certain sums over primes in short interval at \eqref{shortprimesum}.  With greater exertions, one may be able to remove the dependence on GRH and prove Theorems \ref{upperboundsfixedmodmeanS} and \ref{upperboundsfixedmodmeanR} unconditionally. 
  
\section{Preliminaries}
\label{sec 2}

In this section, we collect some auxiliary results needed in the proofs of our theorems.

\subsection{Properties of $\lambda(n)$}
\label{sec:cusp form}

    Recall that $f$ is a fixed holomorphic Hecke eigenform $f$ of weight $\kappa \equiv 0 \pmod 4$ for the full modular group $SL_2 (\mathbb{Z})$. For $\Re(s)>1$, the associated modular $L$-function $L(s, f)$ is defined as
\begin{align*}
L(s, f ) &= \sum_{n=1}^{\infty} \frac{\lambda(n)}{n^s}
 = \prod_{p} \left(1 + \sum^{\infty}_{i=1}\frac{\lambda(p^i)}{p^{is}}\right)=\prod_{p} \left(1 - \frac{\alpha_p }{p^s} \right)^{-1}\left(1 - \frac{\beta_p }{p^s} \right)^{-1}.
\end{align*}
  It follows that $\lambda(n)$ is multiplicative and satisfies the identity
\begin{align}
\label{Lambdapkrel}
 \lambda(p^m)=\sum^{m}_{j=0}\alpha^{m-j}_p\beta^{j}_p,
\end{align}
 for any prime $p$ and any non-negative integer $m$.  Moreover, by Deligne's proof \cite{D} of the Weil conjecture,
\begin{align}
\label{alpha}
|\alpha_{p}|=|\beta_{p}|=1, \quad \alpha_{p}\beta_{p}=1.
\end{align}
From \eqref{Lambdapkrel} and \eqref{alpha}, we deduce
\begin{align} \label{alphalambdarel}
  \alpha_p+\beta_p= \lambda(p) \quad \mbox{and} \quad  \alpha^2_p+\beta^2_p= \lambda^2(p)-2=\lambda(p^2)-1.
\end{align}

 Furthermore, \eqref{Lambdapkrel}, \eqref{alpha} and the multiplicativity of $\lambda(n)$ yield that $\lambda(n) \in \mr$, $\lambda (1) =1$ and for any $\varepsilon>0$,
\begin{align}
\label{lambdabound}
\begin{split}
  |\lambda(n)| \leq d(n) \ll n^{\varepsilon},
\end{split}
\end{align}
where $d(n)$ is the number of positive divisors $n$ and the last bound above follows from \cite[Theorem 2.11]{MVa1}. \newline

 Observe that $|\lambda(n)|^2=\lambda^2(n) \geq 0$ as $\lambda(n)$ is real. It was proved independently by R. A. Rankin \cite{Rankin39} and A. Selberg \cite{Selberg40}  that for $x \geq 1$ and $\theta=3/5$, 
\begin{align}
\label{lambdasquareasymp}
\begin{split}
  \sum_{n \leq x}\lambda^2(n) =\frac {L(1, \operatorname{sym}^2 f)}{\zeta(2)}x+O(x^{\theta}).
\end{split}
\end{align}
  Here $L(s, \operatorname{sym}^2 f)$ is the symmetric square $L$-function of $f$ defined for $\Re(s)>1$ by
 (see \cite[p. 137]{iwakow} and \cite[(25.73)]{iwakow})
\begin{align}
\label{Lsymexp}
\begin{split}
 L(s, \operatorname{sym}^2 f)= \prod_p \frac{1}{1-p^{-s}} \frac{1}{1-\alpha^2_pp^{-s}} \frac{1}{1-\beta^2_pp^{-s}} 
 = \zeta(2s) \sum_{n \geq 1}\frac {\lambda(n^2)}{n^s}=\prod_{p}\left( 1-\frac {\lambda(p^2)}{p^s}+\frac {\lambda(p^2)}{p^{2s}}-\frac {1}{p^{3s}} \right)^{-1}.
\end{split}
\end{align}
 A result of G. Shimura \cite{Shimura} implies that the corresponding completed symmetric square $L$-function
\begin{align}
\label{Lsymfeq}
 \Lambda(s, \operatorname{sym}^2 f)=& \pi^{-3s/2}\Gamma \Big( \frac {s+1}{2} \Big)\Gamma \Big(\frac {s+\kappa-1}{2}\Big) \Gamma \Big(\frac {s+\kappa}{2}\Big) L(s, \operatorname{sym}^2 f)
\end{align}
  is entire and satisfies the functional equation $\Lambda(s, \operatorname{sym}^2 f)=\Lambda(1-s, \operatorname{sym}^2 f)$. Moreover, $L(s, \operatorname{sym}^2 f)$ has no pole at $s=1$. \newline

The current best known $\theta$ in \eqref{lambdasquareasymp} is due to B. Huang \cites{Huang21, Huang24}, who showed that $\theta=3/5-\delta$ with $\delta<3/305$. Moreover, it is known (see \cites{Huang21}), under GRH, for any $\varepsilon>0$, 
\begin{align}
\label{lambdasquareasympconj}
\begin{split}
  \sum_{n \leq x}\lambda^2(n) =\frac {L(1, \operatorname{sym}^2 f)}{\zeta(2)}x+O(x^{1/2+\varepsilon}).
\end{split}
\end{align}

Next,  we include a result on certain sums over primes.
\begin{lemma}
\label{RS}
 Let $x \geq 2$. We have, for some constant $b_1$, $b_2$, $b_3$, $b_4$ with $b_3$, $b_4>0$,
\begin{align}
\label{merten}
\sum_{p\le x} \frac{1}{p} =& \log \log x + b_1+ O\Big(\frac{1}{\log x}\Big), \;  \\
\label{merten1}
\sum_{p\le x} \frac{\lambda^2(p)}{p} =& \log \log x + b_2+ O\Big(\frac{1}{\log x}\Big),  \\
\label{merten2-0}
\sum_{p\le x} \log p =& 
\begin{cases}
  x+O\Big(x \exp(-b_3\sqrt{\log x})\Big), & \text{unconditionally}, \\
  x+O\Big(x^{1/2}\log^2x \Big), &  \text{under GRH}, 
\end{cases} \\
\label{merten2}
\sum_{p\le x} (\lambda^2(p)-1)\log p \ll &  
\begin{cases}
 x \exp(-b_4\sqrt{\log x}), & \text{unconditionally}, \\
 x^{1/2}\log^2x , & \text{under GRH}.
\end{cases} 
\end{align}
\end{lemma}
\begin{proof}
  The expression in \eqref{merten} follows from \cite[Theorem 2.7 (d)]{MVa1}.  The formula \eqref{merten1} follows from \cite[Lemma 2.1]{GHH}.  From \cite[Theorem 6.9]{MVa1} (resp. \cite[Theorem 13.1]{MVa1}), the asymptotic formula in \eqref{merten2-0} holds unconditionally (resp. under GRH). Lastly, it was shown in \cite[(2.9)]{G&Zhao26-04} that the unconditional estimate in \eqref{merten2} is valid. Moreover, it is shown there that if we write, for $\Re(s) \geq 2$,
\begin{align*}
\begin{split}
  -\frac {L'(s, \operatorname{sym}^2 f)}{L(s, \operatorname{sym}^2 f)}=\sum_{n \geq 1}\frac {\Lambda_{\operatorname{sym}^2 f}(n)}{n^s},
\end{split}
\end{align*}
 then by \eqref{alphalambdarel} and \eqref{Lsymexp} that
\begin{align*}
\begin{split}
  \Lambda_{\operatorname{sym}^2 f}(p)=(\alpha^2_p+\beta^2_p+1)\log p=(\lambda^2(p)-1)\log p.
\end{split}
\end{align*}
We apply the above together with \cite[(5.4), (5.8), (5.49), (5.56)]{iwakow} and \eqref{Lsymfeq} to see that conditional bound in \eqref{merten2} is also valid under GRH. This completes the proof of the lemma.
\end{proof}

\subsection{Parseval’s identity}

    The following lemma is taken from \cite[Theorem 5.4]{MVa1} and gives a version of Parseval’s identity for Dirichlet series.
\begin{lemma}
\label{parseval}
    Let $(a_n)_{n\geq 1}$ be a sequence of complex numbers and $F(s)=\sum_{n=1}^{\infty} a_nn^{-s}$ be the corresponding Dirichlet series. If $\sigma_c$ denotes its abscissa of convergence, then, for any $\sigma>\max(0,\sigma_c)$,
    \begin{equation*}
        \int\limits_{1}^{\infty} \frac{\big|\sum_{n\leq x}a_n\big|^2 }{x^{1+2\sigma }} \dif x=\frac{1}{2\pi}\int\limits_{-\infty}^{+\infty}\frac{|F(\sigma+it)|^2}{|\sigma+it|^2} \dif t.
    \end{equation*}
\end{lemma}

\subsection{Expectations on Random Euler products}
  We develope in this section some estimate on the expectations of certain random Euler products.  We start with Steinhaus random multiplicative functions.  Recall that $\mathbb{E}$ denotes the expectation.  We also write $\p$ for the probability measure and $\textbf{1}$ the indicator function throughout the paper. 
\begin{prop}
\label{prop1}
With the notation as above, let $h(n)$ be a Steinhaus random multiplicative function. We have, for any$\alpha$, $\beta \geq 0$, $x$, $y$ with $100(1 + \max\{\alpha^2 , \beta^2\}) \leq x \leq y$, $\sigma \geq - 1/\log y$ and $t \in \mr$,
\begin{align}
\label{Eprod}
\begin{split}
\E \prod_{x < p \leq y} & \left|1 - \frac{\alpha_ph(p)}{p^{1/2+\sigma}}\right|^{-2\alpha} \left|1 - \frac{\beta_ph(p)}{p^{1/2+\sigma}}\right|^{-2\alpha}\left|1 - \frac{\alpha_ph(p)}{p^{1/2+\sigma + it}}\right|^{-2\beta}\left|1 - \frac{\beta_ph(p)}{p^{1/2+\sigma + it}}\right|^{-2\beta}  \\
& = \exp\Big (\sum_{x < p \leq y} \frac{\lambda^2(p)(\alpha^2 + \beta^2 + 2\alpha\beta \cos(t\log p))}{p^{1 + 2\sigma}} + O\Big( \frac{\max\{\alpha, \beta, \alpha^3 , \beta^3\}}{\sqrt{x} \log x} \Big) \Big ) . 
\end{split}
\end{align}
 If additionally $\sigma \leq 1/\log y$, then the left-hand side of \eqref{Eprod} is
\begin{align}
\label{Eprod1}
\begin{split} = \exp \left( O \left(\max\{\alpha, \beta, \alpha^2 , \beta^2\} \left( 1 + \frac{|t|}{\log^{100}x} \right) \right) \right) \left(\frac{\log y}{\log x} \right)^{\alpha^2 + \beta^2} \left(1 + \min\left\{\frac{\log y}{\log x}, \frac{1}{|t|\log x}\right\} \right)^{2\alpha \beta}. 
\end{split}
\end{align}

  In particular, we have for any real $\alpha \geq 0$, any real $100(1 + \alpha^2  \leq x \leq y$, and any real $\sigma \geq - 1/\log y$,
\begin{align}
\label{Eprod2}
\begin{split}  
 \E \prod_{x < p \leq y} \left|1 - \frac{\alpha_ph(p)}{p^{1/2+\sigma}}\right|^{-2\alpha} \left|1 - \frac{\beta_ph(p)}{p^{1/2+\sigma}}\right|^{-2\alpha}=  & \exp\Big (\sum_{x < p \leq y} \frac{\lambda^2(p) \alpha^2}{p^{1 + 2\sigma}} + O\Big(\frac{\max\{\alpha, \alpha^3 \}}{\sqrt{x} \log x} \Big) \Big ) . 
\end{split}
\end{align}
   If $\sigma$ also satisfies $\sigma \leq 1/\log y$, then the above is
\begin{align}
\label{Eprod3}
= \exp \left( O\left( \max\{\alpha, \alpha^2 \} \right) \right) \left(\frac{\log y}{\log x} \right)^{\alpha^2}. 
\end{align}

\end{prop}
\begin{proof}
Our proof is similar to that of the Euler Product Result 2.1 in \cite{harper2020moments}. We set $M = M(\alpha,\beta) := \max\{\alpha, \beta, \alpha^3 , \beta^3\}$.
We note that by \eqref{alpha} and Taylor's approximation,
\begin{align*}
\begin{split}
& \left|1 - \frac{\alpha_ph(p)}{p^{1/2+\sigma}}\right|^{-2\alpha} \left|1 - \frac{\beta_ph(p)}{p^{1/2+\sigma}}\right|^{-2\alpha}\left|1 - \frac{\alpha_ph(p)}{p^{1/2+\sigma + it}}\right|^{-2\beta}\left|1 - \frac{\beta_ph(p)}{p^{1/2+\sigma + it}}\right|^{-2\beta} \\
=&  \exp\Big (-2\alpha \Re\log\Big(1 - \frac{\alpha_ph(p)}{p^{1/2+\sigma}}\Big) -2\alpha \Re\log\Big(1 - \frac{\beta_ph(p)}{p^{1/2+\sigma}}\Big) -2\beta \Re\log\Big(1 - \frac{\alpha_p h(p)}{p^{1/2+\sigma + it}}\Big)-2\beta \Re\log\Big(1 - \frac{\beta_p h(p)}{p^{1/2+\sigma + it}}\Big)\Big )  \\
= & \exp\Big ( \frac{2\alpha \Re (\alpha_p+\beta_p)h(p)}{p^{1/2 + \sigma}} + \frac{\alpha \Re (\alpha^2_p+\beta^2_p)h(p)^2}{p^{1+2\sigma}} +  \frac{2\beta \Re (\alpha_p+\beta_p)h(p)p^{-it}}{p^{1/2 + \sigma}} + \frac{\beta \Re (\alpha^2_p+\beta^2_p)h(p)^2 p^{-2it}}{p^{1+2\sigma}} + O\Big(\frac{\max\{\alpha, \beta\}}{p^{3/2 + 3\sigma}}\Big) \Big ). 
\end{split}
\end{align*}
Next, if $y \geq p > x \geq 100(1+\max\{\alpha^2 , \beta^2\})$, then \eqref{alpha} further implies that every term in the exponential here has size at most $4\max\{\alpha, \beta\}/p^{1/2 + \sigma} = 4\max\{\alpha, \beta\} e^{-\sigma \log p}/ p^{1/2} \leq 2e/5$. Therefore we may apply the Taylor series expansion of the exponential function, finding the above is
\begin{equation*}
\begin{split}
 =  1  + &\frac{2(\alpha \Re (\alpha_p+\beta_p)h(p) + \beta \Re (\alpha_p+\beta_p)h(p)p^{-it})}{p^{1/2 + \sigma}} + \frac{(\alpha \Re (\alpha^2_p+\beta^2_p)h(p)^2 + \beta \Re (\alpha^2_p+\beta^2_p)h(p)^2 p^{-2it})}{p^{1+2\sigma}}  \\
& + \frac{2(\alpha \Re (\alpha_p+\beta_p)h(p) + \beta \Re (\alpha_p+\beta_p)h(p)p^{-it})^2}{p^{1 + 2\sigma}} + O\Big(\frac{M}{p^{3/2 + 3\sigma}}\Big) . 
\end{split}
\end{equation*}
Now taking expectations, by symmetry, we have $\E \Re h(p) = \E \Re h(p)^2 = 0$, similarly for $\E \Re h(p) p^{-it}$ and $\E \Re h(p)^2 p^{-2it}$. A simple trigonometric calculation also shows that $\E (\Re h(p))^2 = 1/2$, and similarly $\E \Re h(p) \Re h(p) p^{-it} = \cos(t\log p)/2$. So by \eqref{alphalambdarel} we get
\begin{align*}
\begin{split}
\E & \left|1 - \frac{\alpha_ph(p)}{p^{1/2+\sigma}}\right|^{-2\alpha} \left|1 - \frac{\beta_ph(p)}{p^{1/2+\sigma}}\right|^{-2\alpha}\left|1 - \frac{\alpha_ph(p)}{p^{1/2+\sigma + it}}\right|^{-2\beta}\left|1 - \frac{\beta_ph(p)}{p^{1/2+\sigma + it}}\right|^{-2\beta} \\
& =  1 + \frac{2 (\E (\alpha \Re (\alpha_p+\beta_p)h(p) + \beta \Re (\alpha_p+\beta_p)h(p)p^{-it})^2)}{p^{1 + 2\sigma}} + O\Big(\frac{M}{p^{3/2 + 3\sigma}}\Big)  \\
& = 1 + \frac{2 (\E (\alpha \Re \lambda(p)h(p) + \beta \Re \lambda(p)h(p)p^{-it})^2)}{p^{1 + 2\sigma}} + O\Big(\frac{M}{p^{3/2 + 3\sigma}}\Big)  \\
& = 1 + \frac{2\lambda^2(p) (\alpha^2 \E (\Re h(p))^2 + 2\alpha\beta \E \Re h(p) \Re h(p) p^{-it} + \beta^2 \E (\Re h(p) p^{-it})^2)}{p^{1 + 2\sigma}} + O\Big(\frac{M}{p^{3/2 + 3\sigma}}\Big)  \\
& =  1 + \frac{\lambda^2(p) (\alpha^2 + \beta^2 + 2\alpha\beta \cos(t\log p))}{p^{1 + 2\sigma}} + O\Big(\frac{M}{p^{3/2 + 3\sigma}}\Big)  = \exp\Big (\frac{\lambda^2(p) (\alpha^2 + \beta^2 + 2\alpha\beta \cos(t\log p))}{p^{1 + 2\sigma}} + O\Big(\frac{M}{p^{3/2 + 3\sigma}}\Big) \Big ) . 
\end{split}
\end{align*}

With the above calculation, the independence of $h$ on distinct primes, and $p^{3/2 + \sigma} = e^{\sigma \log p} p^{3/2} \geq e^{-1} p^{3/2}$ for $p \leq y$, we derive that 
\begin{align*}
\begin{split}
\E \prod_{x < p \leq y} & \left|1 - \frac{\alpha_ph(p)}{p^{1/2+\sigma}}\right|^{-2\alpha} \left|1 - \frac{\beta_ph(p)}{p^{1/2+\sigma}}\right|^{-2\alpha}\left|1 - \frac{\alpha_ph(p)}{p^{1/2+\sigma + it}}\right|^{-2\beta}\left|1 - \frac{\beta_ph(p)}{p^{1/2+\sigma + it}}\right|^{-2\beta} \\
 = & \exp\Big (\sum_{x < p \leq y} \Big(\frac{\lambda^2(p) (\alpha^2 + \beta^2 + 2\alpha\beta \cos(t\log p))}{p^{1 + 2\sigma}} + O\Big(\frac{M}{p^{3/2 + 3\sigma}} \Big) \Big) \Big ). 
\end{split}
\end{align*}
 The above then implies that the validity of  \eqref{Eprod}. \newline
 
 Next, by \eqref{merten2-0} and \eqref{merten2}, we have for some constant $b_4>0$, and any $x \geq 2$,
\begin{align}
\label{merten3}
\begin{split}
\sum_{p\le x} \lambda^2(p)\log p =&  x+O\Big(x \exp(-b_4\sqrt{\log x})\Big).
\end{split}
\end{align}

The above leads to, for $x \geq 2$,
\begin{align}
\label{merten4}
\begin{split}
 \pi_f(x):=\sum_{p\le x} \lambda^2(p)=&  \int\limits^x_2\frac {\dif u}{\log u}+O\Big(x \exp(-b_4\sqrt{\log x})\Big).
\end{split}
\end{align}

  Now, using that $e^{-2\sigma \log p} - 1 \ll |\sigma| \log p \ll \log p/\log y$ for $|\sigma| \leq 1/\log y$, we deduce via \eqref{merten1} and \eqref{merten3} that
\begin{eqnarray*}
\sum_{x < p \leq y} \frac{\lambda^2(p) (\alpha^2 + \beta^2)}{p^{1 + 2\sigma}} & = & (\alpha^2 + \beta^2) \sum_{x < p \leq y} \frac{\lambda^2(p)}{p} + (\alpha^2 + \beta^2) \sum_{x < p \leq y} \frac{\lambda^2(p) (e^{-2\sigma \log p} - 1)}{p} \\
& = & (\alpha^2 + \beta^2)\log\Big(\frac{\log y}{\log x}\Big) + O(\max\{\alpha^2 , \beta^2\}).
\end{eqnarray*}

Similarly,
\begin{align*}
\begin{split}
& \sum_{x < p \leq y} \frac{2\alpha\beta \cos(t\log p)}{p^{1 + 2\sigma}} = \sum_{x < p \leq y} \frac{2\alpha\beta \cos(t\log p)}{p}+ O(\max\{\alpha^2 , \beta^2\}). 
\end{split}
\end{align*}  
Moreover, by \eqref{merten4}, 
\begin{eqnarray*}
\sum_{x < p \leq y} \frac{\lambda^2(p)\cos(t\log p)}{p} = \int\limits_{x}^{y} \frac{\cos(t \log z)}{z} \dif \pi_f(z) = \int\limits_{x}^{y} \frac{\cos(t \log z)}{z \log z} \dif z + O\Big(\frac{1 + |t|}{\log^{100}x} \Big)  = \int\limits_{\log x}^{\log y} \frac{\cos(tu)}{u} \dif u + O\Big(\frac{1+|t|}{\log^{100}x}\Big) . 
\end{eqnarray*}
  We then proceed as in the proof of Euler Product Result 2.1 in \cite{harper2020moments} to get that \eqref{Eprod1} holds. The relations given in \eqref{Eprod2} and \eqref{Eprod3} are obtained by setting $\beta=0$ in \eqref{Eprod} and \eqref{Eprod1}, respectively. This completes the proof of the proposition.  
\end{proof}

Our next result is a counterpart of the above for the Rademacher case. 
\begin{prop}
With the notation as above, let $h(n)$ be a  Rademacher random multiplicative function. We have, for any $\alpha$, $\beta \geq 0$, $x$, $y$ with $100(1 + \max\{\alpha^2 , \beta^2\}) \leq x \leq y$, $\sigma \geq - 1/\log y$ and $t_1$, $t_2 \in \mr$, 
\begin{align}
\label{EprodR}
\begin{split}
 \E \prod_{x < p \leq y} & \left|1 + \frac{\lambda(p)h(p)}{p^{1/2+\sigma + it_1}}\right|^{2\alpha} \left|1 + \frac{\lambda(p)h(p)}{p^{1/2+\sigma + it_2}}\right|^{2\beta} \\
 = & \exp\Big ( \sum_{x < p \leq y} \frac{\lambda^2(p)(\alpha^2 + \beta^2 + (\alpha^2 - \alpha)\cos(2t_1 \log p) + (\beta^2 - \beta)\cos(2t_2 \log p))}{p^{1 + 2\sigma}}  \\
& \hspace*{2cm} + \sum_{x < p \leq y} \frac{\lambda^2(p)(2\alpha\beta(\cos((t_1 + t_2)\log p) + \cos((t_1 - t_2)\log p)))}{p^{1 + 2\sigma}} + O\Big(\frac{\max\{\alpha, \beta, \alpha^3 , \beta^3\}}{\sqrt{x} \log x} \Big) \Big ) .
\end{split}
\end{align} 

 If moreover $\sigma \leq 1/\log y$, then the above is
\begin{align}
\label{Eprod1R}
\begin{split}
=  \exp & \Big( O\Big(\max\{\alpha, \beta, \alpha^2 , \beta^2\} \Big(1 + \frac{|t_1| + |t_2|}{\log^{100}x} \Big)\Big) \Big) \left(1 + \min\left\{\frac{\log y}{\log x}, \frac{|t_1|^{-1}}{\log x}\right\} \right)^{\alpha^2 - \alpha} \left(1 + \min\left\{\frac{\log y}{\log x}, \frac{|t_2|^{-1}}{\log x}\right\} \right)^{\beta^2 - \beta} \\
& \times \left(\frac{\log y}{\log x} \right)^{\alpha^2 + \beta^2} \left( \Big(1 + \min\left\{\frac{\log y}{\log x}, \frac{|t_1 + t_2|^{-1}}{\log x}\right\} \Big) \Big(1 + \min\left\{\frac{\log y}{\log x}, \frac{|t_1 - t_2|^{-1}}{\log x}\right\} \Big) \right)^{2\alpha\beta} . 
\end{split}
\end{align} 

As an upper bound of the above, we may replace the factor $\exp \Big( O\Big(\max\{\alpha, \beta, \alpha^2 , \beta^2\} \Big(1 + \frac{|t_1| + |t_2|}{\log^{100}x} \Big)\Big) \Big)$ by 
\begin{align}
\label{Eprod1R1}
\begin{split}
\exp (O(\max\{\alpha, \beta, \alpha^2 , \beta^2\})) \min\left\{\frac{\log y}{\log x}, 1 + \frac{(|t_1| + |t_2|)^{1/100}}{\log x} \right\}^{|\alpha^2 - \alpha| + |\beta^2 - \beta| + 4\alpha\beta},
\end{split}
\end{align}  
 and as a lower bound we may replace the same by 
 \[ \exp(O(\max\{\alpha, \beta, \alpha^2 , \beta^2\})) \min\left\{\frac{\log y}{\log x}, 1 + \frac{(|t_1| + |t_2|)^{1/100}}{\log x} \right\}^{-(|\alpha^2 - \alpha| + |\beta^2 - \beta| + 4\alpha\beta)}. \]
\end{prop}
\begin{proof}
 The proof proceeds by a modification of that of Euler Product Result 2.2 in \cite{harper2020moments}. We set $M = M(\alpha,\beta) := \max\{\alpha, \beta, \alpha^3 , \beta^3\}$ as well to see from \eqref{alpha} and \eqref{alphalambdarel} that
\begin{equation*}
\begin{split}
& \left|1 + \frac{\lambda(p)h(p)}{p^{1/2+\sigma+it_1}}\right|^{2\alpha}  \left|1 + \frac{\lambda(p)h(p)}{p^{1/2+\sigma + it_2}}\right|^{2\beta}  = \exp\Big (2\alpha \Re\log(1 + \frac{\lambda(p)h(p)}{p^{1/2+\sigma+it_1}}) +2\beta \Re\log(1 + \frac{\lambda(p)h(p)}{p^{1/2+\sigma + it_2}})\Big )  \\
& \hspace*{1cm} = \exp\Big ( \frac{2\alpha \Re \lambda(p)h(p) p^{-it_1}}{p^{1/2 + \sigma}} - \frac{\alpha \Re f(p)^2 p^{-2it_1}}{p^{1+2\sigma}} +  \frac{2\beta \Re f(p)p^{-it_2}}{p^{1/2 + \sigma}} - \frac{\beta \Re f(p)^2 p^{-2it_2}}{p^{1+2\sigma}} + O\Big(\frac{\max\{\alpha, \beta\}}{p^{3/2 + 3\sigma}} \Big) \Big ) \\
& \hspace*{1cm} =  1 + \frac{2(\alpha \Re \lambda(p)h(p)p^{-it_1} + \beta \Re \lambda(p)h(p)p^{-it_2})}{p^{1/2 + \sigma}} - \frac{(\alpha \Re \lambda^2(p)h(p)^2 p^{-2it_1} + \beta \Re \lambda^2(p)h(p)^2 p^{-2it_2})}{p^{1+2\sigma}} + \nonumber \\
& \hspace*{2cm} + \frac{2(\alpha \Re \lambda(p)h(p)p^{-it_1} + \beta \Re \lambda(p)h(p)p^{-it_2})^2}{p^{1 + 2\sigma}} + O\Big(\frac{M}{p^{3/2 + 3\sigma} }\Big) .
\end{split}
\end{equation*}
 Note that in the Rademacher case we have $h(p)^2 \equiv 1$ and $\E \Re h(p) p^{-it} = \cos(t\log p) \E h(p) = 0$. Thus
\begin{equation*}
\begin{split}
\E & \left|1 + \frac{\lambda(p)h(p)}{p^{1/2+\sigma+it_1}}\right|^{2\alpha} \left|1 + \frac{\lambda(p)h(p)}{p^{1/2+\sigma + it_2}}\right|^{2\beta} \nonumber \\
& \hspace*{1cm} = 1 - \frac{\lambda^2(p)(\alpha\cos(2t_1 \log p) + \beta\cos(2t_2 \log p))}{p^{1+2\sigma}} + \frac{2\lambda^2(p) (\alpha\cos(t_1 \log p) + \beta\cos(t_2 \log p) )^2}{p^{1 + 2\sigma}} + O\Big(\frac{M}{p^{3/2 + 3\sigma}} \Big). \nonumber
\end{split}
\end{equation*}
  We then follow the proof of Euler Product Result 2.2 in \cite{harper2020moments} and argue similar to the ones given in the proof of Proposition \ref{prop1} to see that all the assertions of the proposition are valid. This completes the proof.  
\end{proof}

\subsection{Probabilistic Results}
\label{secprobcalc}

  In this section, we gather some probabilistic results on random multiplicative functions, largely inherited from those in \cite{harper2020moments}. For any real $x$, let $\lfloor x \rfloor$ denote the largest integer not exceeding $x$ and $\lceil x \rceil$ the smallest integer greater than or equal to $x$.  For any positive integers $k$, $n$, 
we write $d_{k}(n)$ for the $k$-fold divisor function, which is the number of $k$-tuples of natural numbers whose product equals $n$.  Equivalently, it is the Dirichlet series coefficient of $\zeta(s)^k$, where $\zeta(s)$ is the Riemann zeta function. Our first result is from \cite[Probability Result 2.3]{harper2020moments}.
\begin{prop}
\label{prop3}
With the notation as above, we have for any $q \geq 1$ and sequence of complex numbers $(a_{n})_{n \leq N}$,  
\begin{align*}
\begin{split}
  \E\left|\sum_{n \leq N} a_n f(n) \right|^{2q} \leq 
 \begin{cases}
\displaystyle{\left(\sum_{n \leq N} |a_n|^2 d_{\lceil q \rceil}(n) \right)^{q}}, \quad & \text{if $h(n)$ is a Steinhaus random multiplicative function}, \\ \\
\displaystyle{\left(\sum_{n \leq N} |a_n|^2 d_{2\lceil q \rceil - 1}(n) \right)^{q}}, \quad & \text{if $h(n)$ is a Rademacher random multiplicative function}.
\end{cases}
\end{split}
\end{align*}
\end{prop}

Our next result is an extension of the lower bound part of Khintchine's inequality (see \cite[Lemma 2.8]{G&Zhao26-04}) to both the Steinhaus and the Rademacher case.
\begin{prop}
\label{Khintchine}
     Let $h(n)$ be either a Steinhaus or a Rademacher random multiplicative function and $(a_n)_{n\geq 1}$ a sequence of complex numbers. The we have, for any $q \geq 1$, 
\begin{align}
\label{KTlower}
\begin{split}
 \E \left|\sum_n h(n)\lambda(n)a_n \right|^{2q}
\gg \Big( \sum_n \lambda^2(n)|a_n|^2 \Big)^{q}.
\end{split}
\end{align}
\end{prop}

  Let $(\Omega, \mathcal{F}, \p)$ be a probability space. We say $(\mathcal{F}_{n})_{n \geq 0}$ is a filtration on $\mathcal{F}$ if it is a sequence of sub-$\sigma$-algebras satisfying $\mathcal{F}_{0} \subseteq \mathcal{F}_{1} \subseteq ... \subseteq \mathcal{F}$. We say a sequence of random variables $(X_n)_{n \geq 0}$ on $(\Omega, \mathcal{F}, \p)$ is a submartingale relative to $(\mathcal{F}_{n})_{n \geq 0}$ and $\p$ if it satisfies:
\begin{enumerate}
\item $X_n$ is measurable with respect to $\mathcal{F}_{n}$, for all $n \geq 0$;

\item  $\E|X_n|$ is finite, for all $n \geq 0$;

\item the conditional expectation $\E(X_n |\mathcal{F}_{n-1})$ is $\geq X_{n-1}$ almost surely for all $n \geq 1$, .
\end{enumerate}

  The following result is Doob’s maximal inequality from \cite[Theorem 9.4]{gut}. 
\begin{prop}
\label{prop4}
 With the notation as above, let $(X_n)_{n \geq 0}$ be a non-negative submartingale on some probability space and with respect to some filtration. Then for any $p > 1$, we have
$$ \E \Big( \max_{0 \leq k \leq n} X_k \Big)^p \leq \Big(\frac{p}{p-1} \Big)^p \E X_{n}^p . $$
\end{prop}

\section{Connecting to Euler products}\label{seceulerprod}
 
For a Steinhaus or Rademacher random multiplicative function $h(n)$, any $P \geq 2$ and $s \in \comc$ with $\Re(s) > 0$, let $F_{P, f}$ denote the partial Euler product of $h(n)\lambda(n)$ over $P$-smooth numbers.  Recall that for any positive real $x$, a positive integer is said to be $x$-smooth if all of its prime factors do not exceed $x$. Thus we have by \eqref{Lambdapkrel},
\begin{align*}
\begin{split}
F_{P,f}(s) =& \sum_{\substack{n=1, \\ n \; \text{is} \; P \; \text{smooth}}}^{\infty} \frac{h(n)\lambda(n)}{n^s} \\
=& \begin{cases}
\displaystyle \prod_{p \leq P} \left(1 - \frac{\alpha_ph(p)}{p^s}\right)^{-1}\left(1 - \frac{\beta_ph(p)}{p^s}\right)^{-1},  \quad & \text{$h$ is a Steinhaus random multiplicative function}, \\
\displaystyle \prod_{p \leq P} \left(1 + \frac{\lambda(p)h(p)}{p^s} \right), \quad & \text{$h$ is a Rademacher random multiplicative function}.
\end{cases}
\end{split}
\end{align*}
  We also write, for brevity, $F_{k,f}$ for the function $F_{x^{e^{-(k+1)}}, f}$ with any integer $-1 \leq k \leq \log\log x $ and $F_f$ for $F_{-1,f}$.  We further set $\| \cdot \|_{r} := (\E|\cdot|^{r})^{1/r}$ for $r \geq 1$. \newline
 
We now follow the treatments of Section 4 in \cite{harper2020moments} to bound $\E|\sum_{n \leq x} h(n)\lambda(n)|^{2q}$ for $q \geq 1$ by connectiing to integrals of Euler products.

\subsection{Upper bounds}
\label{secupperbounds}

The following result gives the upper bounds of $\E|\sum_{n \leq x} h(n)\lambda(n)|^{2q}$ and is analoguous to \cite[Propositions 4.1--4.2]{harper2020moments}.
\begin{prop}
\label{propstupper}
With the notation as above, let $x$ be large and $\mathcal{L} := \lfloor (\log\log x)/10 \rfloor$. Then we have, uniformly for all $1 \leq q \leq \log^{0.05}x$ and any Steinhaus random multiplicative function $h(n)$,
\begin{align}
\label{EupperboundS}
\begin{split}
 \Big\| \sum_{n \leq x} h(n)\lambda(n) \Big\|_{2q} \leq \sqrt{\frac{x}{\log x}} e^{O(q)} \sum_{0 \leq k \leq \mathcal{L}} \Big\| \int\limits_{-1/2}^{1/2} |F_{k,f}( \tfrac12 + \tfrac{q-k}{\log x} + it)|^2 \dif t \Big\|_{q}^{1/2} + e^{O(q)} \sqrt{\frac{x}{\log x}}.
\end{split}
\end{align}
 Similarly, uniformly for all $1 \leq q \leq \log^{0.05}x$ and any Rademacher random multiplicative function $h(n)$,
\begin{align}
\label{EupperboundR}
\begin{split}
\Big\| \sum_{n \leq x} h(n)\lambda(n) \Big\|_{2q} \leq  \sqrt{\frac{x}{\log x}} e^{O(q)} \sum_{0 \leq k \leq \mathcal{L}} & \max_{N \in \mz} \frac{1}{(|N|+1)^{1/8}} \Big\| \int\limits_{N-1/2}^{N+1/2} |F_{k,f}(\tfrac12 + \tfrac{q-k}{\log x} + it)|^2 \dif t \Big\|_{q}^{1/2} \\
& + e^{O(q)} \sqrt{\frac{x}{ \log x}}.
\end{split}
\end{align}
\end{prop}
\begin{proof}
  We denote $P(n)$ the largest prime factor of $n$. For $h(n)$ being either a Steinhaus or a Rademacher random multiplicative function, Minkowski's inequality gives that, for all $q \geq 1$,
\begin{align}
\label{EPndecomp}
\Big\| \sum_{n \leq x} h(n)\lambda(n) \Big\|_{2q} \leq \sum_{0 \leq k \leq \mathcal{L}} \Big\| \sum_{\substack{n \leq x, \\ x^{e^{-(k+1)}} < P(n) \leq x^{e^{-k}}}} h(n)\lambda(n) \Big\|_{2q} + \Big\| \sum_{\substack{n \leq x, \\ P(n) \leq x^{e^{-(\mathcal{L}+1)}} }} h(n)\lambda(n) \Big\|_{2q} .
 \end{align}
  
 We now apply Proposition \ref{prop3} and bound $\textbf{1}_{n \leq x}$ by $(\frac{x}{n})^{1-1/\log^{0.9}x} = xe^{-\log^{0.1}x} /n^{1-1/\log^{0.9}x}$ via Rankin's trick to see that
\begin{align}
\label{sumdsmallsmoothS}
\begin{split} 
\Big\| \sum_{\substack{n \leq x, \\ P(n) \leq x^{e^{-(\mathcal{L}+1)}} }} h(n)\lambda(n) \Big\|_{2q} \leq \sqrt{\sum_{\substack{n \leq x, \\ P(n) \leq x^{e^{-(\mathcal{L}+1)}} }} d_{\lceil q \rceil}(n)\lambda^2(n) } \leq \sqrt{x e^{-\log^{0.1}x} \sum_{\substack{n \\ P(n) \leq x^{e^{-(\mathcal{L}+1)}} }} \frac{d_{\lceil q \rceil}(n)\lambda^2(n)}{n^{1 - 1/\log^{0.9}x}} }. 
\end{split}
\end{align}
Now,
\begin{align}
\label{sumdlambdapsmall}
\begin{split}
\sum_{\substack{n \\ P(n) \leq x^{e^{-(\mathcal{L}+1)}} }} & \frac{d_{\lceil q \rceil}(n)\lambda^2(n)}{n^{1 - 1/\log^{0.9}x}}  = \prod_{p \leq x^{e^{-(\mathcal{L}+1)}}} \Big(1 - \frac{\lambda^2(p)}{p^{1 - 1/\log^{0.9}x}}\Big)^{- \lceil q \rceil}\prod_{p \leq x^{e^{-(\mathcal{L}+1)}}}\Big(1 - \frac{\lambda^2(p)}{p^{1 - 1/\log^{0.9}x}}\Big)^{ \lceil q \rceil} \\
& \hspace*{5cm} \times \prod_{p \leq x^{e^{-(\mathcal{L}+1)}}}\Big(1+\frac{d_{\lceil q \rceil}(p)\lambda^2(p)}{p^{1 - 1/\log^{0.9}x}}+\sum^{\infty}_{j=2}\frac{d_{\lceil q \rceil}(p^j)\lambda^2(p^j)}{p^{j(1 - 1/\log^{0.9}x)}}\Big) \\
& = \prod_{p \leq x^{e^{-(\mathcal{L}+1)}}} \Big(1 - \frac{\lambda^2(p)}{p^{1 - 1/\log^{0.9}x}}\Big)^{- \lceil q \rceil}\prod_{p \leq x^{e^{-(\mathcal{L}+1)}}}\Big(1-\frac{d_{\lceil q \rceil}(p)\lambda^2(p)}{p^{1 - 1/\log^{0.9}x}}+O\Big(\frac{\lceil q \rceil^2}{p^{2(1 - 1/\log^{0.9}x)}}\Big)\Big) \\
& \hspace*{2cm} \times \prod_{p \leq x^{e^{-(\mathcal{L}+1)}}}\Big(1+\frac{d_{\lceil q \rceil}(p)\lambda^2(p)}{p^{1 - 1/\log^{0.9}x}}+O\Big(\frac{d_{\lceil q \rceil}(p^2)}{p^{2(1 - 1/\log^{0.9}x)}}\Big)\Big) \\
& = \prod_{p \leq x^{e^{-(\mathcal{L}+1)}}}\Big(1 - \frac{\lambda^2(p)}{p^{1 - 1/\log^{0.9}x}}\Big)^{- \lceil q \rceil}\prod_{p \leq x^{e^{-(\mathcal{L}+1)}}}\Big(1+O\Big(\frac{d_{\lceil q \rceil}(p^2)}{p^{2(1 - 1/\log^{0.9}x)}}\Big)\Big).
\end{split}
\end{align}

  Recall that $\mathcal{L} := \lfloor (\log\log x)/10 \rfloor$, so that $p^{1/\log^{0.9}x} \ll \exp(\log x \cdot e^{-\log \log x/10}/\log^{0.9}x) \ll 1$ for $p \leq x^{e^{-(\mathcal{L}+1)}}$. It follows that
\begin{align*}
\begin{split}
  \prod_{p \leq x^{e^{-(\mathcal{L}+1)}}}\Big(1+O\Big(\frac{d_{\lceil q \rceil}(p^2)}{p^{2(1 - 1/\log^{0.9}x)}}\Big)\Big)=\prod_{p \leq x^{e^{-(\mathcal{L}+1)}}}\Big(1+O\Big(\frac{d_{\lceil q \rceil}(p^2)}{p^2}\Big)\Big) \ll 1.
\end{split}
\end{align*}

From \eqref{sumdlambdapsmall} and the above, we get
\begin{align}
\label{sumdlambdapsmall1}
\begin{split}
   \sum_{\substack{n \\ P(n) \leq x^{e^{-(\mathcal{L}+1)}} }} \frac{d_{\lceil q \rceil}(n)\lambda^2(n)}{n^{1 - 1/\log^{0.9}x}} \ll &  \prod_{p \leq x^{e^{-(\mathcal{L}+1)}}}\Big(1 - \frac{\lambda^2(p)}{p^{1 - 1/\log^{0.9}x}}\Big)^{- \lceil q \rceil} \ll \prod_{p \leq e^{\log^{0.9}x}}\Big(1 - \frac{\lambda^2(p)}{p^{1 - 1/\log^{0.9}x}}\Big)^{- \lceil q \rceil} \\
= & \exp \Big(-\lceil q \rceil \sum_{p \leq e^{\log^{0.9}x}} \log \Big(1 - \frac{\lambda^2(p)}{p^{1 - 1/\log^{0.9}x}} \Big) \Big) = e^{O(q)}\exp \Big(\lceil q \rceil \sum_{p \leq e^{\log^{0.9}x}}\frac{\lambda^2(p)}{p^{1 - 1/\log^{0.9}x}} \Big) \\
= & e^{O(q)}\exp \Big(O(q) \sum_{p \leq e^{\log^{0.9}x}}\frac{\lambda^2(p)}{p} \Big)  = \log^{O(q)}x,
\end{split}
\end{align}
 where the last bound above follows from \eqref{merten1}. As $q \leq \log^{0.05}x$, from \eqref{sumdsmallsmoothS}--\eqref{sumdlambdapsmall1}, for a sufficiently small $c>0$, we have
\begin{align}
\label{sumdsmallsmoothS1}
\begin{split} 
\Big\| \sum_{\substack{n \leq x, \\ P(n) \leq x^{e^{-(\mathcal{L}+1)}} }} h(n)\lambda(n) \Big\|_{2q} \ll \sqrt{x} e^{-c\log^{0.1}x}. 
\end{split}
\end{align}

 Next, we denote by $\E^{(k)}$ the expectation conditional on $(h(p))_{p \leq x^{e^{-(k+1)}}}$.  Using Proposition \ref{prop3}, we obtain
\begin{align}
\label{sumdPnsmall}
\begin{split} 
 \sum_{0 \leq k \leq \mathcal{L}} & \Big\| \sum_{\substack{n \leq x, \\ x^{e^{-(k+1)}} < P(n) \leq x^{e^{-k}}}} h(n)\lambda(n) \Big\|_{2q} 
=  \sum_{0 \leq k \leq \mathcal{L}} \Big\| \sum_{\substack{1 < m \leq x , \\ p|m \Rightarrow x^{e^{-(k+1)}} < p \leq x^{e^{-k}}}} h(m)\lambda(m) \sum_{\substack{n \leq x/m, \\ n \; \text{is} \; x^{e^{-(k+1)}} \text{-smooth}}} h(n)\lambda(n)  \Big\|_{2q}  \\
& = \sum_{0 \leq k \leq \mathcal{L}} \Biggl( \E \E^{(k)}\Biggl|\sum_{\substack{1 < m \leq x , \\ p|m \Rightarrow x^{e^{-(k+1)}} < p \leq x^{e^{-k}}}} h(m)\lambda(m) \sum_{\substack{n \leq x/m, \\ n \; \text{is} \; x^{e^{-(k+1)}} \text{-smooth}}} h(n)\lambda(n)  \Biggr|^{2q} \Biggr)^{1/2q}  \\
& \leq \sum_{0 \leq k \leq \mathcal{L}} \Big( \E \Big( \sum_{\substack{1 < m \leq x , \\ p|m \Rightarrow x^{e^{-(k+1)}} < p \leq x^{e^{-k}}}} d_{\lceil q \rceil}(m)\lambda^2(m) \Big|\sum_{\substack{n \leq x/m, \\ n \; \text{is} \; x^{e^{-(k+1)}} \text{-smooth}}} h(n)\lambda(n)  \Big|^2 \Big)^{q} \Big)^{1/2q} . 
\end{split}
\end{align}

Setting $X = e^{\sqrt{\log x}}$, we then see that, uniformly for any $1 \leq q \leq \log^{0.05}x$, the above is
\begin{align}
\label{uppersmoothingdisplay}
\begin{split}
 =  \sum_{0 \leq k \leq \mathcal{L}} & \Big\| \sum_{\substack{1 < m \leq x , \\ p|m \Rightarrow x^{e^{-(k+1)}} < p \leq x^{e^{-k}}}} d_{\lceil q \rceil}(m)\lambda^2(m) \Big|\sum_{\substack{n \leq x/m, \\ n \; \text{is} \; x^{e^{-(k+1)}} \text{-smooth}}} h(n)\lambda(n)  \Big|^2 \Big\|_{q}^{1/2} \\
& \ll  \sum_{0 \leq k \leq \mathcal{L}} \Big\| \sum_{\substack{1 < m \leq x , \\ p|m \Rightarrow x^{e^{-(k+1)}} < p \leq x^{e^{-k}}}} d_{\lceil q \rceil}(m)\lambda^2(m) \frac{X}{m} \int\limits_{m}^{m(1 + 1/X)} \Big| \sum_{\substack{n \leq x/t, \\ x^{e^{-(k+1)}} \text{-smooth}}} h(n)\lambda(n)  \Big|^2 \dif t \Big\|_{q}^{1/2} \\
& \hspace*{1cm} + \sum_{0 \leq k \leq \mathcal{L}} \Big\| \sum_{\substack{1 < m \leq x , \\ p|m \Rightarrow x^{e^{-(k+1)}} < p \leq x^{e^{-k}}}} d_{\lceil q \rceil}(m)\lambda^2(m) \frac{X}{m} \int\limits_{m}^{m(1 + 1/X)} \Big| \sum_{\substack{x/t < n \leq x/m, \\ x^{e^{-(k+1)}} \text{-smooth}}} h(n)\lambda(n)  \Big|^2 \dif t \Big\|_{q}^{1/2} .
\end{split}
\end{align}

 We now apply Minkowski's inequality, followed by H\"{o}lder's inequality with exponent $q$ to the normalized integral $\frac{X}{m} \int_{m}^{m(1 + 1/X)} \dif t$ to see that the second term on the right-hand side of \eqref{uppersmoothingdisplay} is
\begin{align}
\label{secondtermest}
\begin{split}
 \leq & \sum_{0 \leq k \leq \mathcal{L}} \sqrt{\sum_{\substack{1 < m \leq x , \\ p|m \Rightarrow x^{e^{-(k+1)}} < p \leq x^{e^{-k}}}} d_{\lceil q \rceil}(m)\lambda^2(m) \Big\| \frac{X}{m} \int\limits_{m}^{m(1 + 1/X)} \Big| \sum_{\substack{x/t < n \leq x/m, \\ x^{e^{-(k+1)}} \text{-smooth}}} h(n)\lambda(n) \Big|^2 \dif t \Big\|_{q} } \\
 \leq & \sum_{0 \leq k \leq \mathcal{L}} \sqrt{\sum_{\substack{1 < m \leq x , \\ p|m \Rightarrow x^{e^{-(k+1)}} < p \leq x^{e^{-k}}}} d_{\lceil q \rceil}(m)\lambda^2(m) \Biggl( \frac{X}{m} \int\limits_{m}^{m(1 + 1/X)} \E \Biggl| \sum_{\substack{x/t < n \leq x/m, \\ x^{e^{-(k+1)}} \text{-smooth}}} h(n)\lambda(n) \Biggr|^{2q} \dif t \Biggr)^{1/q} }  \\
= & \sum_{0 \leq k \leq \mathcal{L}} \sqrt{M_1+M_2} . 
\end{split}
\end{align}
  where $M_1$ denotes the sum under the square root in the penultimate line in \eqref{secondtermest} with $m$ restricted to $1 < m \leq Y$ and $M_2$ the complementary sum.  We shall specify the value of $Y$ later. \newline

Now Cauchy's inequality renders 
\begin{align}
\label{E2qMsmall}
\begin{split}
 \E \Biggl| \sum_{\substack{x/t < n \leq x/m, \\ x^{e^{-(k+1)}} \text{-smooth}}} h(n)\lambda(n) \Biggr|^{2q} \leq & \sqrt{\E \Biggl| \sum_{\substack{x/t < n \leq x/m, \\ x^{e^{-(k+1)}} \text{-smooth}}} h(n)\lambda(n) \Biggr|^{2} 
\E \Biggl| \sum_{\substack{x/t < n \leq x/m, \\ x^{e^{-(k+1)}} \text{-smooth}}} h(n)\lambda(n) \Biggr|^{2(2q-1)}},
\end{split}
\end{align}
 where the last estimation above follows by arguing as those in Section \ref{secsteineasy} to see that
\begin{align}
\label{E2q-1}
\begin{split} 
\E \Biggl| \sum_{\substack{x/t < n \leq x/m, \\ x^{e^{-(k+1)}} \text{-smooth}}} h(n)\lambda(n) \Biggr|^{2(2q-1)} \ll  \Big( \frac {x}{m} \Big)^{2q-1} \log^{O(q^2)}x. 
\end{split}
\end{align}  

  On the other hand, we have by \eqref{lambdasquareasympconj} that under GRH, 
\begin{align*}
\begin{split}
 & \E \Biggl| \sum_{\substack{x/t < n \leq x/m, \\ x^{e^{-(k+1)}} \text{-smooth}}} h(n)\lambda(n) \Biggr|^{2} =  \sum_{\substack{x/t < n \leq x/m, \\ x^{e^{-(k+1)}} \text{-smooth}}}\lambda^2(n) \ll \sum_{\substack{x/t < n \leq x/m}}\lambda^2(n) \ll \frac xm-\frac xt+\left(\frac x{m}\right)^{1/2+\varepsilon} \ll \frac {x}{mX}+ \left(\frac x{m}\right)^{1/2+\varepsilon} ,
\end{split}
\end{align*}
where the last bound follows from the observation that if  $m \leq t \leq m(1+1/X)$, then
\[ \frac xm-\frac xt \leq \frac xm-\frac x{m(1+1/X)} \ll \frac {x}{mX}. \]

Now $\left( x/m\right)^{1/2+\varepsilon} \leq \frac {x}{mX}$ when $m \leq x/X^{2/(1-2\varepsilon)}$, so that we have for $1 < m \leq x/X^{2/(1-2\varepsilon)}$, 
\begin{align}
\label{EPmlarge2}
\begin{split}
\E \Biggl| \sum_{\substack{x/t < n \leq x/m, \\ x^{e^{-(k+1)}} \text{-smooth}}} h(n)\lambda(n) \Biggr|^{2} \ll \frac {x}{mX}.  
\end{split}
\end{align}
  We deduce from the above, \eqref{E2qMsmall} and \eqref{E2q-1} that for $1 < m \leq x/X^{2/(1-2\varepsilon)}$, 
\begin{align}
\label{E2qMsmallest}
\begin{split}
 \E \Biggl| \sum_{\substack{x/t < n \leq x/m, \\ x^{e^{-(k+1)}} \text{-smooth}}} h(n)\lambda(n) \Biggr|^{2q} \ll & \sqrt{\frac {x}{mX}  \Big(\frac{x}{m} \Big)^{2q-1} \log^{O(q^2)}x}.
\end{split}
\end{align}

 On the other hand, by a result of M. N. Huxley \cite{Huxley03} on the Dirichlet divisor problem , we have, 
\begin{equation*} 
  \sum_{n \leq x}d(n) = x\log x+(2\gamma_0-1)+O(x^{131/416+\varepsilon}),
\end{equation*}
where $\gamma_0$ is the Euler constant. From this, \eqref{lambdabound} and \eqref{EPmlarge2} for $x/X^{2/(1-2\varepsilon)} < m \leq x$, 
\begin{align*}
\begin{split}
 \Biggl| \sum_{\substack{x/t < n \leq x/m, \\ x^{e^{-(k+1)}} \text{-smooth}}} h(n)\lambda(n) \Biggl| \leq \sum_{\substack{x/t < n \leq x/m}} d(n) & \leq 
\frac {x}{m}\log \frac xm-\frac {x}{t}\log \frac xt+O \Big( \Big(\frac xm \Big)^{131/416+\varepsilon} \Big) \\
&  \ll \Big( \frac {x}{m}-\frac {x}{t} \Big)\log \frac xm+\Big(\frac xm\Big)^{1/2+\varepsilon},
\end{split}
\end{align*}
using the mean value theorem. This implies that for $x/X^{2/(1-2\varepsilon)} < m \leq x$, 
\begin{align*}
\begin{split}
 \Biggl| \sum_{\substack{x/t < n \leq x/m, \\ x^{e^{-(k+1)}} \text{-smooth}}} h(n)\lambda(n) \Biggl| \ll \frac {x}{Xm}\log \frac xm+\Big( \frac xm \Big)^{1/2+\varepsilon} \ll \Big( \frac xm \Big)^{1/2+\varepsilon}.
\end{split}
\end{align*}
  We thus conclude that for $x/X^{2/(1-2\varepsilon)} < m \leq x$, 
\begin{align}
\label{E2qMlargeest}
\begin{split}
 \E \Biggl| \sum_{\substack{x/t < n \leq x/m, \\ x^{e^{-(k+1)}} \text{-smooth}}} h(n)\lambda(n) \Biggr|^{2q} \ll & \Big(\frac xm \Big)^{q + 2 q \varepsilon}.
\end{split}
\end{align}
 
  We deduce from \eqref{secondtermest} upon setting $Y = x/X^{2/(1-2\varepsilon)}$, \eqref{E2qMsmallest} and \eqref{E2qMlargeest} that the second term in \eqref{uppersmoothingdisplay} is 
\begin{align}
\label{secondtermest1}
\begin{split}
 \ll & \sum_{0 \leq k \leq \mathcal{L}} \sqrt{\frac{x \log^{O(q)}x}{X^{1/2q}} \sum_{\substack{1 < m \leq Y , \\ p|m \Rightarrow x^{e^{-(k+1)}} < p \leq x^{e^{-k}}}} \frac{d_{\lceil q \rceil}(m)\lambda^2(m)}{m} + \sum_{\substack{Y < m \leq x , \\ p|m \Rightarrow x^{e^{-(k+1)}} < p \leq x^{e^{-k}}}} d_{\lceil q \rceil}(m)\lambda^2(m)\Big(\frac xm\Big)^{1+\varepsilon} } . 
\end{split}
\end{align}

Now
\begin{align}
\label{Eulerprodest}
\begin{split} 
\sum_{\substack{1 < m \leq Y , \\ p|m \Rightarrow x^{e^{-(k+1)}} < p \leq x^{e^{-k}}}} \frac{d_{\lceil q \rceil}(m)\lambda^2(m)}{m} \leq \prod_{x^{e^{-(k+1)}} < p \leq x^{e^{-k}}}\Big(1+\frac{d_{\lceil q \rceil}(p)\lambda^2(p)}{p}+\sum^{\infty}_{j=2}\frac{d_{\lceil q \rceil}(p^j)\lambda^2(p^j)}{p^{j}} \Big) . 
\end{split}
\end{align}  

  Recall that for any positive integer $k$,  $\sum_{n=1}^{\infty} \frac{d_{k}(n)}{n^s} = \prod_{p}(1 - p^{-s})^{-k}$, which implies that for any integer $a\geq 1$ and any prime $p$, we have
\begin{align}
\label{dpest}
\begin{split}
  d_k(p^a)=(-1)^a\binom {-k}{a}=\binom {a+k-1}{a}=\binom {a+k-1}{k-1}. 
\end{split}
\end{align}
  It follows that
\begin{align*}
\begin{split}
 d_k(p^a) \leq 
\displaystyle \begin{cases}
\frac {(2k-3)^{a}}{a!}, \quad & \text{if $a < k-1$}, \\
\frac {(2a)^{k-1}}{(k-1)!}, \quad & \text{if $a \geq k-1$}.
\end{cases}
\end{split}
\end{align*}
This, from the above and \eqref{lambdabound},
\begin{align}
\label{ppowersum}
\begin{split}
  \sum^{\infty}_{j=2}\frac{d_{\lceil q \rceil}(p^j)\lambda^2(p^j)}{p^{j}}=O\Big(\frac{d_{\lceil q \rceil}(p^2)}{p^{2}} \Big). 
\end{split}
\end{align}

  We note moreover infer from \eqref{dpest} that $d_{\lceil q \rceil}(m) \leq \lceil q \rceil^{\Omega(m)}$, where $\Omega(m)$ denotes the number of distinct prime powers dividing $m$. Using this and \eqref{ppowersum}, we deduce from \eqref{Eulerprodest} that
\begin{align}
\label{Eulerprodest1}
\begin{split} 
\sum_{\substack{1 < m \leq Y , \\ p|m \Rightarrow x^{e^{-(k+1)}} < p \leq x^{e^{-k}}}} \frac{d_{\lceil q \rceil}(m)\lambda^2(m)}{m} = e^{O(q)}. 
\end{split}
\end{align}  

  We further apply the estimation $d_{\lceil q \rceil}(m) \leq \lceil q \rceil^{\Omega(m)}$ together with \eqref{lambdabound} to see that $\lambda(m)^2 \leq d^2(m) \leq 4^{\Omega(m)}$. We also note that if $m \geq x^{1/20}$ only has prime factors from the interval $(x^{e^{-(k+1)}}, x^{e^{-k}}]$, then we must have $\Omega(m) \geq e^{k}/20$. We now apply Number Theory Result 1 in \cite{Harper20} by setting the parameters there to be $v=v, u=x^{1/20}, y=x^{e^{-(k+1)}}, z=x^{e^{-k}}, \delta = 1/2$.  For $v \geq x^{1/10}$, we get
\begin{align*}
\begin{split}
\sum_{\substack{x/X \leq m \leq x , \\ p|m \Rightarrow x^{e^{-(k+1)}} < p \leq x^{e^{-k}}}} & d_{\lceil q \rceil}(m)\lambda^2(m)
\leq (5^{20})^{-e^{k}/20}\sum_{\substack{x^{1/20} \leq m \leq v, \\ p | m \Rightarrow x^{e^{-(k+1)}} \leq p \leq x^{e^{-k}}}} (4 \cdot 5^{20}\lceil q \rceil)^{\Omega(m)} \\
 \ll & 5^{-e^{k}} \frac{qv}{\log(x^{e^{-(k+1)}})} \prod_{x^{e^{-(k+1)}} \leq p \leq x^{e^{-k}}}\left(1 - \frac{4 \cdot 5^{20}\lceil q \rceil}{p}\right)^{-1}  \ll e^{O(q)}2^{-e^{k}} \frac{v}{\log x} .
\end{split}
\end{align*}
 where the last estimation above follows from \eqref{merten}. \newline

The above and partial summation to see that
\begin{align}
\label{EEjremainder21-1}
\begin{split}
  \sum_{\substack{Y < m \leq x, \\ p | m \Rightarrow x^{e^{-(k+1)}} \leq p \leq x^{e^{-k}}}} d_{\lceil q \rceil}(m) \lambda^2(m) \Big(\frac xm \Big)^{1+\varepsilon} 
  \ll e^{O(q)}2^{-e^{k}} \frac{x}{\log x}.
\end{split}
\end{align}

  As we have $q \leq \log^{0.05}x$, $\mathcal{L} = \lfloor (\log\log x)/10 \rfloor$ and $X = e^{\sqrt{\log x}}$, we see from \eqref{secondtermest1}, \eqref{Eulerprodest1} and \eqref{EEjremainder21-1} that the second term on the right-hand side of \eqref{uppersmoothingdisplay} is
\begin{align}
\label{Secondtermest}
\begin{split} \leq \sum_{0 \leq k \leq \mathcal{L}} \sqrt{\frac{x \log^{O(q)}x}{X^{1/2q}} + e^{O(q)} 2^{-e^{k}} \frac{x}{\log x} } \leq e^{O(q)} \sqrt{\frac{x}{\log x}}. 
\end{split}
\end{align}

 It remains to treat the first sum in \eqref{uppersmoothingdisplay}, which equals to
$$ \sum_{0 \leq k \leq \mathcal{L}} \Big\| \int\limits_{x^{e^{-(k+1)}}}^{x} \Big| \sum_{\substack{n \leq x/t, \\ x^{e^{-(k+1)}} \text{-smooth}}} h(n)\lambda(n) \Big|^2 \sum_{\substack{t/(1+1/X) \leq m \leq t , \\ p|m \Rightarrow x^{e^{-(k+1)}} < p \leq x^{e^{-k}}}} \frac{X}{m} d_{\lceil q \rceil}(m)\lambda^2(m) \dif t \Big\|_{q}^{1/2} . $$

  We set $u = u(k,t) := e^{k} (\log t)/\log x$ to see that if $m \geq t/(1+1/X)$ only has prime factors from the interval $(x^{e^{-(k+1)}}, x^{e^{-k}}]$ then we must have $\Omega(m) \geq u$. Using again the inequality $d_{\lceil q \rceil}(m) \leq \lceil q \rceil^{\Omega(m)}, \lambda^2(m) \leq 4^{\Omega(m)}$,  from Number Theory Result 1 in \cite{Harper20} and \eqref{merten}, for $x$ large enough, we have
\begin{eqnarray*}
\sum_{\substack{t/(1+1/X) \leq m \leq t , \\ p|m \Rightarrow x^{e^{-(k+1)}} < p \leq x^{e^{-k}}}} \frac{X}{m} d_{\lceil q \rceil}(m)\lambda^2(m) & \ll & \frac{X}{t} 5^{-u} \sum_{\substack{t/(1+1/X) \leq m \leq t , \\ p|m \Rightarrow x^{e^{-(k+1)}} < p \leq x^{e^{-k}}}} (20 \lceil q \rceil)^{\Omega(m)} \\
& \ll & \frac{q e^k 5^{-u}}{\log x} \prod_{x^{e^{-(k+1)}} < p \leq x^{e^{-k}}} \Big(1-\frac{20 \lceil q \rceil}{p}\Big)^{-1} \ll \frac{e^{O(q)}}{\log t}. 
\end{eqnarray*}

 We then derive that the first sum in \eqref{uppersmoothingdisplay} is
\begin{align*}
\leq & e^{O(q)} \sum_{0 \leq k \leq \mathcal{L}} \Big\| \int\limits_{x^{e^{-(k+1)}}}^{x} \Biggl| \sum_{\substack{n \leq x/t, \\ x^{e^{-(k+1)}} \text{-smooth}}} h(n)\lambda(n) \Biggr|^2 \frac{\dif t}{\log t} \Big\|_{q}^{1/2} \\
= & e^{O(q)} \sum_{0 \leq k \leq \mathcal{L}} \sqrt{x} \Big\| \int\limits_{1}^{x^{1 - e^{-(k+1)}}} \Biggl| \sum_{\substack{n \leq z, \\ x^{e^{-(k+1)}} \text{-smooth}}} h(n)\lambda(n) \Biggr|^2 \frac{\dif z}{z^2 \log(x/z)} \Big\|_{q}^{1/2} ,
\end{align*}
where the last expression above follows from a the substitution $z=x/t$ in the first integral above. \newline

We then argue as in the proof of Propositions 4.1--4.2 in \cite{Harper23}.  Upon applying Lemma \ref{parseval} and Minkowski's inequality to see that the first sum in \eqref{uppersmoothingdisplay} is
\begin{align}
\label{Firsttermest}
\begin{split} 
\leq & \frac{\sqrt{x}}{\sqrt{\log x}} e^{O(q)} \sum_{0 \leq k \leq \mathcal{L}} \Big\| \int\limits_{1}^{x^{1 - e^{-(k+1)}}} \Biggl| \sum_{\substack{n \leq z, \\ x^{e^{-(k+1)}} \text{-smooth}}} f(n) \Biggr|^2 \frac{\dif z}{z^{2 + 2q/\log x - 2k/\log x}} \Big\|_{q}^{1/2}  \\
 \leq & \sqrt{\frac{x}{\log x}} e^{O(q)} \sum_{0 \leq k \leq \mathcal{L}} \Big\| \int\limits_{-\infty}^{\infty} \frac{|F_{k,f}(\tfrac12+\tfrac{q}{\log x} - \tfrac{k}{\log x} + it)|^2}{|\tfrac12+\tfrac{q}{\log x} - \tfrac{k}{\log x}  + it|^2} \dif t \Big\|_{q}^{1/2}  \\
 \leq & \sqrt{\frac{x}{\log x}} e^{O(q)} \sum_{0 \leq k \leq \mathcal{L}} \sqrt{\sum_{n \in \mz} \frac{1}{n^2 + 1} \Big\| \int\limits_{n-1/2}^{n+1/2} |F_{k,f}( \tfrac12+\tfrac{q-k}{\log x} + it)|^2 \dif t \Big\|_{q}}. 
\end{split}
\end{align}

  Note that the law of the random function $h(n)$ is the same as the law of $h(n)n^{it}$ for any fixed $t \in \mr$ for the Steinhaus case. Thus we have
\[ \Big\| \int\limits_{n-1/2}^{n + 1/2} |F_{k,f}(\tfrac12+\tfrac{q-k}{\log x} + it)|^2 \dif t \Big\|_{q} = \Big\| \int\limits_{-1/2}^{1/2} |F_{k,f}(\tfrac12+\tfrac{q-k}{\log x} + it)|^2 \dif t \Big\|_{q} \;\;\; \mbox{for all} \; n . \]

  We now conclude from \eqref{EPndecomp}, \eqref{sumdsmallsmoothS1}--\eqref{uppersmoothingdisplay}, \eqref{Secondtermest}, \eqref{Firsttermest} and the above that \eqref{EupperboundS} is valid. The proof of \eqref{EupperboundR} is very similar to the Steinhaus case above. We use the Rademacher part of Proposition \ref{prop3}, producing various terms $d_{2\lceil q \rceil - 1}(n)$ in place of $d_{\lceil q \rceil}(n)$  in the arguments above. We then follow the arguments until the very end, at which point we apply
\begin{eqnarray*}
 \sqrt{\sum_{n \in \mz} \frac{1}{n^2 + 1} \Big\| \int\limits_{n-1/2}^{n+1/2} |F_{k,f}(\tfrac12+\tfrac{q-k}{\log x} + it)|^2 \dif t \Big\|_{q}} \ll  \sqrt{\max_{N \in \mz} \frac{1}{(|N|+1)^{1/4}} \Big\| \int\limits_{N-1/2}^{N+1/2} |F_{k,f}(\tfrac12+\tfrac{q-k}{\log x} + it)|^2 \dif t \Big\|_{q} }  
\end{eqnarray*}
 to arrive at \eqref{EupperboundR}. This completes the proof of the proposition.
\end{proof}

\subsection{Lower bounds}
 
 Recall the function $F_{f}$ defined earlier so that $F_f$ is the partial Euler product of $h(n)\lambda(n)$ over $x$-smooth numbers.  Analoguous to \cite[Proposition 4.3--4.4]{harper2020moments}, we have the following result concerning the lower bounds of $\E|\sum_{n \leq x} h(n)\lambda(n)|^{2q}$. 
\begin{prop}
\label{propstlower}
 With the notation as above, let $h(n)$ be a Steinhaus random multiplicative function or a Rademacher random multiplicative function. We have uniformly for all $q \geq 1$ and large real $x$,
\begin{align}
\label{hlambdalower}
\begin{split}
\Big\| \sum_{n \leq x} h(n)\lambda(n) \Big\|_{2q} \gg \sqrt{\frac{x}{\log x}} \Big\| \int\limits_{1}^{x^{1/4}} \left|\sum_{m \leq z} h(m)\lambda(m) \right|^2 \frac{\dif z}{z^{2}} \Big\|_{q}^{1/2}.
\end{split}
\end{align}
 
 In particular, for any large fixed constant $V \leq (\log x)/q$ and any Steinhaus random multiplicative function $h(n)$, there is an absolute constant $C > 0$ such that we have 
\begin{align}
\label{hlambdalower1}
\begin{split}
 \Big\| \sum_{n \leq x} h(n)\lambda(n) \Big\|_{2q} \gg \sqrt{\frac{x}{\log x}} \Biggl( \Big\| \int_{-1/2}^{1/2} |F_f(\tfrac 12 + \tfrac{4Vq}{\log x} + it)|^2 \dif t \Big\|_{q}^{1/2} - \frac{C}{e^{Vq/2}} \Big\| \int\limits_{-1/2}^{1/2} |F_f(\tfrac 12 + \tfrac{2Vq}{\log x} + it)|^2 \dif t \Big\|_{q}^{1/2}  \Biggr).
\end{split}
\end{align}
 Similarly,  for any Rademacher random multiplicative function $h(n)$, we have 
\begin{align}
\label{hlambdalower2}
\begin{split}
\Big\| \sum_{n \leq x} h(n)\lambda(n) \Big\|_{2q}  \gg \sqrt{\frac{x}{\log x}} & \Biggl( \Big\| \int\limits_{-1/2}^{1/2} |F_f(\tfrac12 + \tfrac{4Vq}{\log x} + it)|^2 \dif t \Big\|_{q}^{1/2}  \\
& - \frac{C}{e^{Vq/2}} \max_{N \in \mz} \frac{1}{(|N|+1)^{1/8}} \Big\| \int\limits_{N-1/2}^{N+1/2} |F_f(\tfrac12 + \tfrac{2Vq}{\log x} + it)|^2 \dif t \Big\|_{q}^{1/2}   \Biggr) .
\end{split}
\end{align}
\end{prop}
\begin{proof}
 Recall that $P(n)$ denotes the largest prime factor of $n$. Let $\epsilon$ be an auxiliary Rademacher random variable that is independent of everything else. We argue in a manner similar to the Proof of Propositions 4.3 and 4.4. in \cite{harper2020moments} and arrive at, $q \geq 1$,
\begin{eqnarray*}
\Big\| \sum_{n \leq x} h(n)\lambda(n) \Big\|_{2q}  \geq  \Big\| \sum_{\substack{n \leq x, \\ P(n) > x^{3/4}}} h(n)\lambda(n) \Big\|_{2q}.
\end{eqnarray*}

 Note that $\sum_{\substack{n \leq x, \\ P(n) > x^{3/4}}} h(n)\lambda(n) = \sum_{x^{3/4} < p \leq x} h(p)\lambda(p) \sum_{m \leq x/p} h(m)\lambda(m)$. Now the inner sums are determined by the values $(h(p)\lambda(p))_{p \leq x^{1/4}}$, which are independent of $(h(p)\lambda(p))_{x^{3/4} < p \leq x}$.  Thus applying \eqref{KTlower} with $a_p = \sum_{m \leq x/p} h(m)\lambda(m)$ renders
\begin{align*}
\begin{split} 
\Big\| \sum_{\substack{n \leq x, \\ P(n) > x^{3/4}}} h(n)\lambda(n) \Big\|_{2q} \geq & \Big\| \sum_{x^{3/4} < p \leq x} \lambda^2(p) \Big| \sum_{m \leq x/p} h(m)\lambda(m) \Big|^2 \Big\|_{q}^{1/2} \\
\geq & \frac{1}{\sqrt{\log x}} \Big\| \sum_{x^{3/4} < p \leq x} \lambda^2(p)\log p \Big|\sum_{m \leq x/p} h(m)\lambda(m) \Big|^2 \Big\|_{q}^{1/2} . 
\end{split}
\end{align*}

Now setting $r=\lfloor x/p \rfloor$, we get
\begin{align}
\label{hlambdalower3}
\begin{split} \sum_{x^{3/4} < p \leq x} \lambda^2(p)\log p \left|\sum_{m \leq x/p} h(m)\lambda(m) \right|^2 = \sum_{r \leq x^{1/4}} \sum_{x/(r+1) < p \leq x/r} \lambda^2(p)\log p \left|\sum_{m \leq r} h(m)\lambda(m) \right|^2. 
\end{split}
\end{align}
Here we observe that $x/r - x/(r+1) = x/(r(r+1)) \gg (x/r)^{2/3}$ for $r \leq x^{1/4}$. Moreover, from \eqref{merten2-0} and \eqref{merten2}, under GRH,  
\begin{align*}
\begin{split} 
\sum_{p \leq x}\lambda^2(p)\log p=x+O(x^{1/2}\log^2x). 
\end{split}
\end{align*}
The above lead to, for $r \leq x^{1/4}$, 
\begin{align}
\label{shortprimesum}
\begin{split} 
\sum_{x/(r+1) < p \leq x/r} \lambda^2(p)\log p \gg \int\limits_{x/(r+1)}^{x/r} 1 \dif t. 
\end{split}
\end{align}
The above and \eqref{hlambdalower3} yield
$$ \sum_{x^{3/4} < p \leq x} \lambda^2(p)\log p \left|\sum_{m \leq x/p} h(m)\lambda(m) \right|^2 \gg \sum_{r \leq x^{1/4}} \int\limits_{x/(r+1)}^{x/r} 1 \ \dif t \left|\sum_{m \leq r} h(m)\lambda(m) \right|^2 \geq \int\limits_{x^{3/4}}^{x} \left|\sum_{m \leq x/t} h(m)\lambda(m) \right|^2 \dif t . $$
 We now make a substitution $z = x/t$ in the last integral above to see that it equals $x \int_{1}^{x^{1/4}} \left|\sum_{m \leq z} h(m)\lambda(m) \right|^2 \dif z/z^2$. This implies the validity of \eqref{hlambdalower}. \newline

The estimations given in \eqref{hlambdalower1} and \eqref{hlambdalower2} are then proved using  \eqref{hlambdalower} and using Lemma \ref{parseval}, by following the arguments in the proof of Propositions 4.3 and 4.4 in \cite{harper2020moments}. This completes the proof of the proposition.
\end{proof}

\section{Proof of Theorems \ref{upperboundsfixedmodmeanS} and \ref{upperboundsfixedmodmeanR}}

To establish Theorems \ref{upperboundsfixedmodmeanS} and \ref{upperboundsfixedmodmeanR}, it suffices to prove \eqref{mainestimationupper} and \eqref{mainestimationlower}. We start by proving \eqref{mainestimationupper}.  To this end, we divide into two cases according to the size of $q$.  

\subsection{Proof of \eqref{mainestimationupper}, large $q$}
\label{seceasycases}

In this section, we prove \eqref{mainestimationupper} for $\log\log x \leq q \leq c\log x/\log\log x$. 

\subsubsection{Upper bounds, Steinhaus case}
\label{secsteineasy}
 As shown in the proof of \cite[The upper bound in the Steinhaus case, for very large $q$]{harper2020moments}, to establish \eqref{mainestimationupper} for the Steinhaus case when $\log\log x \leq q \leq c\log x/\log\log x$, it suffices to show that
\begin{align}
\label{UpperboundqlargeS}
\begin{split} 
 \Big\| \sum_{n \leq x} h(n)\lambda(n) \Big\|_{2q} \leq \exp \Big( -\tfrac{q}{2} \log q - \tfrac{q}{2} \log\log(2q) + O(q) \Big) \sqrt{x} \log^{q/2}x.  
\end{split}
\end{align}
 To prove this, we first apply Minkowski's inequality to obtain that
\begin{align*}
\begin{split} 
 \Big\| \sum_{n \leq x} h(n)\lambda(n) \Big\|_{2q} = \Big\| \sum_{\substack{m \leq x, \\ m \; \text{is} \; q^2 \; \text{smooth}}} h(m)\lambda(m) \sum_{\substack{n \leq x/m, \\ p \mid n \Rightarrow p > q^2}} h(n)\lambda(n) \Big\|_{2q} \leq \sum_{\substack{m \leq x, \\ m \; \text{is} \; q^2 \; \text{smooth}}}  \lambda^2(m)\Big\| \sum_{\substack{n \leq x/m, \\ p \mid n \Rightarrow p > q^2}} h(n)\lambda(n) \Big\|_{2q} . 
\end{split}
\end{align*}
 Using Proposition \ref{prop3}, and then using Rankin's trick of upper bounding $\textbf{1}_{n \leq x/m}$ by $(\frac{x}{nm})^{1+q/\log x}$, we get
\begin{eqnarray*}
\Big\| \sum_{\substack{n \leq x/m, \\ p \mid n \Rightarrow p > q^2}} h(n)\lambda(n) \Big\|_{2q} \leq \Biggl (\sum_{\substack{n \leq x/m, \\ p \mid n \Rightarrow p > q^2}} d_{\lceil q \rceil}(n)\lambda^2(n) \Biggr)^{1/2} & \leq & \Biggl ( \Big(\frac{x}{m}\Big)^{1 + q/\log x} \sum_{\substack{n \\ p \mid n \Rightarrow p > q^2}} \frac{d_{\lceil q \rceil}(n)\lambda^2(n)}{n^{1+q/\log x}} \Biggr)^{1/2}.
\end{eqnarray*}

  Note that
\begin{align*}
\begin{split}
  \sum_{\substack{p \mid n \Rightarrow p > q^2}} \frac{d_{\lceil q \rceil}(n)\lambda^2(n)}{n^{1+q/\log x}}=\prod_{p > q^2}\Big(1+\frac{d_{\lceil q \rceil}(p)\lambda^2(p)}{p^{1+q/\log x}}+\sum^{\infty}_{j=2}\frac{d_{\lceil q \rceil}(p^j)\lambda^2(p^j)}{p^{j(1+q/\log x)}}\Big). 
\end{split}
\end{align*}

Now this, \eqref{lambdabound} and \eqref{dpest} render that
\begin{align*}
\begin{split}
  \sum^{\infty}_{j=2}\frac{d_{\lceil q \rceil}(p^j)\lambda^2(p^j)}{p^{j(1+q/\log x)}}=O\Big(\frac{d_{\lceil q \rceil}(p^2)}{p^{2(1+q/\log x)}}\Big). 
\end{split}
\end{align*}

  Similarly, we have
\begin{align*}
\begin{split}
  \prod_{p > q^2}\Big(1 - \frac{\lambda^2(p)}{p^{1 + q/\log x}}\Big)^{ \lceil q \rceil}=1-\frac{d_{\lceil q \rceil}(p)\lambda^2(p)}{p^{1 + q/\log x}}+O\Big(\frac{\lceil q \rceil^2}{p^{2(1+q/\log x)}}\Big). 
\end{split}
\end{align*}
  It follows, after computation similar to that in \eqref{sumdlambdapsmall},
\begin{align}
\label{sumdlambdaplarge}
\sum_{\substack{p \mid n \Rightarrow p > q^2}} \frac{d_{\lceil q \rceil}(n)\lambda^2(n)}{n^{1+q/\log x}} 
= \prod_{p > q^2}\Big(1 - \frac{\lambda^2(p)}{p^{1 + q/\log x}} \Big)^{- \lceil q \rceil}\prod_{p > q^2}\Big(1+O\Big(\frac{d_{\lceil q \rceil}(p^2)}{p^{2(1+q/\log x)}}\Big)\Big).
\end{align}

  Using the well-known inequality $1+x \leq e^x$ for all real $x$, we see that
\begin{align*}
\begin{split}
 \prod_{p > q^2}\Big(1+O\Big(\frac{d_{\lceil q \rceil}(p^2)}{p^{2(1+q/\log x)}}\Big)\Big) \leq \exp \Big(\sum_{p > q^2}\Big(O\frac{d_{\lceil q \rceil}(p^2)}{p^{2(1+q/\log x)}}\Big)\Big) \ll 1.
\end{split}
\end{align*}

  We thus conclude that
\begin{align}
\label{hlambdaupper}
\begin{split}
 \Big\| \sum_{\substack{n \leq x/m, \\ p \mid n \Rightarrow p > q^2}} h(n)\lambda(n) \Big\|_{2q} \leq & \sqrt{\frac{x}{m}} e^{O(q)} \prod_{p > q^2}\Big(1 - \frac{\lambda^2(p)}{p^{1 + q/\log x}}\Big)^{- \lceil q \rceil/2} \\
=&\sqrt{\frac{x}{m}} e^{O(q)} \prod_{p }\Big(1 - \frac{\lambda^2(p)}{p^{1 + q/\log x}}\Big)^{- \lceil q \rceil/2}\prod_{p \leq q^2}\Big(1 - \frac{\lambda^2(p)}{p^{1 + q/\log x}}\Big)^{\lceil q \rceil/2}.
\end{split}
\end{align}
In the interest of brevity, we write
\begin{align}
\label{Fsdef}
\begin{split}
 F(s)=\prod_{p }\Big(1 - \frac{\lambda^2(p)}{p^{s}}\Big)^{- 1}, \quad \mbox{for} \; \Re (s) >1.
\end{split}
\end{align}
  Then from \eqref{lambdabound}, for $\Re(s)>1$, 
\begin{align*}
\begin{split}
 F(s)=\zeta(s)\prod_{p }\Big(1 - \frac{1}{p^{s}}\Big)\Big(1 - \frac{\lambda^2(p)}{p^{s}}\Big)^{- 1} & =\zeta(s)\prod_{p }\Big(1 - \frac{1}{p^{s}}\Big)\Big(1 + \frac{\lambda^2(p)}{p^{s}}+O\Big(\frac 1{p^{2s}}\Big)\Big) \\
 & =\zeta(s)\prod_{p }\Big(1 + \frac{\lambda^2(p)-1}{p^{s}}+O\Big(\frac 1{p^{2s}}\Big)\Big).
\end{split}
\end{align*}
  Now, by \eqref{alphalambdarel}, we see that $\lambda^2(p)-1=\lambda(p^2)$, so that by \eqref{lambdabound} and \eqref{Lsymexp} we have
\begin{align}
\label{Fsdecomp}
\begin{split}
 F(s)= \zeta(s)\prod_{p }\Big(1 + \frac{\lambda(p^2)}{p^{s}}+O\Big(\frac 1{p^{2s}}\Big)\Big)= & \frac {\zeta(s) L(s, \operatorname{sym}^2 f)}{\zeta(2s)}\prod_{p}\left( 1-\frac {\lambda(p^2)}{p^s}+\frac {\lambda(p^2)}{p^{2s}}-\frac {1}{p^{3s}} \right)\Big(1 + \frac{\lambda(p^2)}{p^{s}}+O\Big(\frac 1{p^{2s}}\Big)\Big) \\
=& \frac {\zeta(s) L(s, \operatorname{sym}^2 f)}{\zeta(2s)}\prod_{p}\Big(1+O\Big(\frac {1}{p^{2s}}\Big)\Big). 
\end{split}
\end{align}
  We substitute the above into \eqref{hlambdaupper} to see that
\begin{align}
\label{hlambdaupper1}
\begin{split}
 \Big\| \sum_{\substack{n \leq x/m, \\ p \mid n \Rightarrow p > q^2}} h(n)\lambda(n) \Big\|_{2q} \leq &\sqrt{\frac{x}{m}} e^{O(q)} \frac {\zeta(1 + q/\log x)^{\lceil q \rceil/2} L(1 + q/\log x, \operatorname{sym}^2 f)^{\lceil q \rceil/2}}{\zeta(2(1 + q/\log x))^{\lceil q \rceil/2}} \\
  & \hspace*{2cm} \times \prod_{p}\Big(1+O\Big(\frac {1}{p^{2(1 + q/\log x)}}\Big)\Big)^{\lceil q \rceil/2}\prod_{p \leq q^2}\Big(1 - \frac{\lambda^2(p)}{p^{1 + q/\log x}}\Big)^{\lceil q \rceil/2} \\
= &\sqrt{\frac{x}{m}} e^{O(q)} \frac {\zeta(1 + q/\log x)^{\lceil q \rceil/2} L(1 + q/\log x, \operatorname{sym}^2 f)^{\lceil q \rceil/2}}{\zeta(2(1 + q/\log x))}\prod_{p \leq q^2}\Big(1 - \frac{\lambda^2(p)}{p^{1 + q/\log x}}\Big)^{\lceil q \rceil/2}.
\end{split}
\end{align}  

It follows from the proof of \cite[Lemma 2.4]{G&Zhao25-11} that
\begin{align}
\label{Lsymbound}
\begin{split}
 L(1 + \tfrac q{\log x}, \operatorname{sym}^2 f) \ll 1.
\end{split}
\end{align}  
  Also, from \cite[Corollary 1.7]{MVa1} and \cite[Theorem 6.7]{MVa1}, we deduce that
\begin{align}
\label{zeta2sbound1}
\begin{split}
   \zeta(1 + q/\log x) = \frac {\log x}{q}+O(1), \quad \frac {1}{\zeta(2(1 + q/\log x))} \ll  1. 
\end{split}
\end{align}
  We apply \eqref{Lsymbound} and \eqref{zeta2sbound1} to see from \eqref{hlambdaupper1} that
\begin{align*}
\begin{split}
 & \Big\| \sum_{\substack{n \leq x/m, \\ p \mid n \Rightarrow p > q^2}} h(n)\lambda(n) \Big\|_{2q} 
\leq \sqrt{\frac{x}{m}} e^{O(q)} \Big(\frac {\log x}{q}\Big)^{\lceil q \rceil/2}\prod_{p \leq q^2}\Big(1 - \frac{\lambda^2(p)}{p^{1 + q/\log x}}\Big)^{\lceil q \rceil/2}.
\end{split}
\end{align*}  
Now
\begin{align*}
\begin{split}
 \prod_{p \leq q^2}\Big(1 - \frac{\lambda^2(p)}{p^{1 + q/\log x}}\Big)^{\lceil q \rceil/2}=\exp \Big(\frac {\lceil q \rceil}2\sum_{p \leq q^2}\Big(1- \frac{\lambda^2(p)}{p^{1 + q/\log x}}\Big)\Big)=e^{O(q)} \exp \Big(-\frac {\lceil q \rceil}2\sum_{p \leq q^2}\frac{\lambda^2(p)}{p^{1 + q/\log x}} \Big).
\end{split}
\end{align*}  
  We may now take the small constant $c$ in the statement of Theorem \ref{upperboundsfixedmodmeanS} to satisfy $c<1$ so that we have $q\log p <\log x$ for $\log\log x \leq q \leq \frac{c\log x}{\log\log x}, p \leq q^2$. It follows from this, \eqref{lambdabound} and \eqref{merten1} that 
\begin{align*}
\begin{split}
 \sum_{p \leq q^2}\frac{\lambda^2(p)}{p^{1 + q/\log x}}=& \sum_{p \leq q^2}\frac{\lambda^2(p)}{p}+\sum_{p \leq q^2}\frac{\lambda^2(p)}{p}(e^{-q\log p/\log x}-1)  \\
 = & \log \log q+O\Big(1+\sum_{p \leq q^2}\frac{q\log p}{p\log x}\Big) = \log \log q+O\Big(1+\frac{q\log q}{\log x}\Big),
\end{split}
\end{align*}  
 where the last estimation above follows from \eqref{merten2-0} and partial summation. We observe that $q\log q/\log x \ll 1$ when $\log\log x \leq q \leq c\log x/\log\log x$ so that we have
\begin{align}
\label{hlambdaupper4}
\begin{split}
 \sum_{p \leq q^2}\frac{\lambda^2(p)}{p^{1 + q/\log x}}
=& \log \log q+O(1). 
\end{split}
\end{align}  
From \eqref{hlambdaupper1}--\eqref{hlambdaupper4}
\begin{align*}
\begin{split}
 & \Big\| \sum_{\substack{n \leq x/m, \\ p \mid n \Rightarrow p > q^2}} h(n)\lambda(n) \Big\|_{2q} 
\leq \sqrt{\frac{x}{m}} e^{O(q)} \Big(\frac {\log x}{q\log q}\Big)^{\lceil q \rceil/2}.
\end{split}
\end{align*}  
  We now sum the above over $m$ by noting that 
  \[ \sum_{\substack{m \leq x, \\ m \; \text{is} \; q^2 \; \text{smooth}}} \frac{\lambda^2(m)}{\sqrt{m}} \leq \exp \Big( \sum_{p \leq q^2} O(1/\sqrt{p}) \Big) \leq \exp\Big( O(q) \Big). \]
  This leads to the estimation given in \eqref{UpperboundqlargeS}. 

\subsubsection{Upper bounds, Rademacher case}
 Similar to the Steinhaus case, in order to prove  \eqref{mainestimationupper} for the Rademacher case when $\log\log x \leq q \leq c\log x/\log\log x$, it suffices to show that, for Rademacher random multiplicative $h(n)$, we have
\begin{align}
\label{UpperboundqlargeR}
\begin{split}  \Big\| \sum_{n \leq x} h(n)\lambda(n) \Big\|_{2q} \leq \exp \Big( -q\log q - q\log\log(2q) + O(q) \Big) \sqrt{x} \log^{q}x. 
\end{split}
\end{align} 

Using Minkowski's inequality, Proposition \ref{prop3}, and then Rankin's trick, we get
\begin{align*}
\begin{split}
\Big\| \sum_{n \leq x} h(n)\lambda(n) \Big\|_{2q}  \leq & \sum_{\substack{m \leq x, \\ m \; \text{is} \; q^2 \; \text{smooth}}}  \lambda^2(m)\Big\| \sum_{\substack{n \leq x/m, \\ p \mid n \Rightarrow p > q^2}} h(n)\lambda(n) \Big\|_{2q} \leq \sum_{\substack{m \leq x, \\ m \; \text{is} \; q^2 \; \text{smooth}}} \lambda^2(m)\Biggl (\sum_{\substack{n \leq x/m, \\ p \mid n \Rightarrow p > q^2}}d_{2\lceil q \rceil - 1}(n)\lambda^2(n)  \Biggr)^{1/2} \\
 \leq & \sum_{\substack{m \leq x, \\ m \; \text{is} \; q^2 \; \text{smooth}}} \sqrt{\frac{x}{m}}\lambda^2(m) e^{O(q)} \Biggl (\sum_{\substack{ n \\ p \mid n \Rightarrow p > q^2}} \frac{d_{2\lceil q \rceil-1}(n)\lambda^2(n)}{n^{1+q/\log x}}\Biggr)^{1/2}. 
\end{split}
\end{align*} 
 We now estimate the last expression above in a way similar to the Steinhaus case.  This lead to the bound in \eqref{UpperboundqlargeR}.

\subsection{Proof of \eqref{mainestimationupper}, small $q$}
\label{seccasesqsmall}

Note that the estimation given in \eqref{mainestimationupper} is valid when $q=1$, by evaluating $\E |\sum_{n \leq x} h(n)\lambda(n)|^{2}$ directly via expanding out the square, using the independence of $h(p)$ and \eqref{lambdasquareasymp}. We may thus assume that $q>1$ in what follows. \newline

  Similar to what is shown in \cite[Section 5]{harper2020moments}, we see that in order to \eqref{mainestimationupper}, it suffices to establish the following result. 
\begin{prop}
 With the notation as above and let $q_0 = (1 + \sqrt{5})/2$. For all large $x$, we have uniformly for $1 \leq q \leq \log^{100}x$ and $-1 \leq k \leq \mathcal{L} = \lfloor (\log\log x)/10 \rfloor$ and $- \frac{e^k}{\log x} \leq \sigma \leq \frac{1}{100\log(2q)}$, 
\begin{align}
\label{Ehalfq}
\begin{split} 
 \E\Big( \int\limits_{-1/2}^{1/2} |F_{k,f}(\tfrac12 + \sigma + it)|^2 \dif t \Big)^q \ll 
\begin{cases}
\frac{e^{O(q^2)}}{\log^{q-1}x} \Big( \frac{\log x}{\log 2q} \Big)^{q^2} \min\Big\{\frac{1}{e^{k+1}}, \frac{1}{|\sigma| \log x} \Big\}^{q^2 - q + 1}, \quad &
\text{if $h$ is Steinhaus}, \\
   \frac{e^{O(q^2)}}{\log^{q}x} \Big(1 + \min\Big\{\log\log x, \frac{1}{|q-q_0|} \Big\} \Big) \Big( \frac{\log x}{\log 2q} \Big)^{\max\{2q^2 - q, q^2 +1\}} & \\
\hspace*{1cm} \times \min\Big\{\frac{1}{e^{k+1}}, \frac{1}{|\sigma| \log x} \Big\}^{\max\{2q^2 - 2q, q^2 - q + 1\}}, \quad & \text{if $h$ is Rademacher}.
\end{cases}
\end{split}
\end{align}
  Furthermore, for any Rademacher random multiplicative function $h(n)$, we have for $|N| \geq 1$, 
\begin{align}
\label{EhalfqNlargeR}
\begin{split} 
\E\Big( \int\limits_{N-\frac{1}{2}}^{N+\frac{1}{2}} |F_{k,f}(\tfrac{1}{2} + \sigma + it)|^2 \dif t \Big)^q  \ll & \min\Big\{|N|^{\frac{1}{100}}, \frac{\log x}{e^{k+1} \log 2q}, \frac{1}{|\sigma| \log 2q}\Big\}^{q(q+1)}  \\
& \hspace*{1cm} \times \frac{e^{O(q^2)}}{\log^{q-1}x} \Big( \frac{\log x}{\log 2q} \Big)^{q^2} \min\Big\{\frac{1}{e^{k+1}}, \frac{1}{|\sigma| \log x} \Big\}^{q^2 - q + 1} .
\end{split}
\end{align}
\end{prop}
\begin{proof}
 We set $\mathcal{X} := \min\{ \log x/e^{k+1}, 1/|\sigma| \}$, and note that we have $\mathcal{X} \geq 100\log(2q)$ under our hypotheses. As shown in the section ``Preliminary Maneuvres" of \cite{harper2020moments}, we see that for $q \geq 1$ and $h(n)$ being either a Steinhaus or a Rademacher random multiplicative function, 
\begin{eqnarray}
\Big\| \int\limits_{-1/2}^{1/2} |F_{k,f}(\tfrac12+\sigma + it)|^2 \dif t \Big\|_{q} \leq \frac{1}{\mathcal{X}} \Big( \int\limits_{-1/(2\mathcal{X})}^{1/(2\mathcal{X})} \mathcal{X} \; \E \Big(\sum_{|n| \leq \mathcal{X}/2 + 1} |F_{k,f}(\tfrac12+\sigma + i(\tfrac{n}{\mathcal{X}} + t))|^2 \Big)^q \dif t \Big)^{1/q} . \nonumber
\end{eqnarray}

If $h(n)$ is a Steinhaus random multiplicative function, the above further simplifies to 
\begin{align}
\label{upperq-1}
\begin{split} 
\Big\| \int\limits_{-1/2}^{1/2} |F_{k,f}(\tfrac12+\sigma + it)|^2 \dif t \Big\|_{q}  \ll \frac{1}{\mathcal{X}} \Big( \mathcal{X} \; \E |F_{k,f}(\tfrac{1}{2}+\sigma)|^2 \Big( \sum_{|m| \leq \mathcal{X}} |F_{k,f}(\tfrac{1}{2}+\sigma + \tfrac{im}{\mathcal{X}})|^2 \Big)^{q-1} \Big)^{1/q}.
\end{split}
\end{align}

Moreover, if $h(n)$ is a Rademacher random multiplicative function, then, for any integer $N$, 
\begin{align}
\label{radupperq-1}
\begin{split} 
 \Big\| & \int\limits_{N-1/2}^{N+1/2} |F_{k,f}(\tfrac{1}{2}+\sigma + it)|^2 \dif t \Big\|_{q} \\
\leq & \frac{1}{\mathcal{X}} \Big(\int\limits_{-1/(2\mathcal{X})}^{1/(2\mathcal{X})} \mathcal{X} \; \E \sum_{|n| \leq \mathcal{X}/2 + 1} |F_{k,f}(\tfrac{1}{2}+\sigma + i(\tfrac{n}{\mathcal{X}} + N + t))|^2  \Big( \sum_{|m| \leq \mathcal{X}/2 + 1} |F_{k,f}(\tfrac{1}{2}+\sigma + i(\frac{m}{\mathcal{X}} + N + t))|^2 \Big)^{q-1} \dif t \Big)^{1/q} . 
\end{split}
\end{align}

  We now proceed to prove \eqref{Ehalfq} by considering the cases $q \geq 2$ and $1\leq q<2$ separately.
\newline
\newline
{\it i) The case $q \geq 2$.}
\newline
\newline
  We argue as in the Proof of Key Proposition 5.1 for $q \geq 2$ in \cite{harper2020moments}.  If $q \geq 2$, we have for $h$ being Steinhaus, 
\begin{align}
\label{holderbigq-1} 
\begin{split} 
\E |F_{k,f}(\tfrac{1}{2}+\sigma)|^2 & \Big( \sum_{|m| \leq \mathcal{X}} |F_{k,f}(\tfrac{1}{2}+\sigma + \tfrac{im}{\mathcal{X}})|^2 \Big)^{q-1} \\
\leq &e^{O(q)} \sum_{|m| \leq \mathcal{X}} \frac{1}{(|m|+1)^2}\cdot (|m|+1)^{2(q-1)} \E |F_{k,f}( \tfrac12+\sigma)|^{2} |F_{k,f}(1/2+\sigma + \tfrac{im}{\mathcal{X}})|^{2(q-1)} .
\end{split}
\end{align}

 We now use the definition of $F_{k,f}(s)$ and apply the trivial bound 
 \[ \exp \Big( O\Big(\sum_{p \leq 100q^2} \frac{q}{p^{1/2+\sigma}} \Big) \Big) = \exp \Big( O\Big( \sum_{p \leq 100q^2} \frac{q}{\sqrt{p}} \Big) \Big)) = \exp \Big( O \Big(\frac{q^2}{\log q } \Big) \Big) \]
  in view of $- e^k/\log x \leq \sigma \leq 1/(100\log(2q))$ to see that
\begin{align}
\label{upperq-2}
\begin{split}  
&  (|m|+1)^{2(q-1)} \E |F_k(1/2+\sigma)|^{2} |F_{k,f}(1/2+\sigma + \frac{im}{\mathcal{X}})|^{2(q-1)} \\
& = \exp \Big( O \Big(\frac{q^2}{\log q } \Big) \Big)  (|m|+1)^{2(q-1)} \\
& \hspace*{0.5cm} \times \E \prod_{100q^2 < p \leq x^{e^{-(k+1)}}} \left|1 - \frac{\alpha_ph(p)}{p^{1/2+\sigma}}\right|^{-2}\left|1 - \frac{\beta_ph(p)}{p^{1/2+\sigma}}\right|^{-2}  \left|1 - \frac{\alpha_ph(p)}{p^{1/2+\sigma + im/\mathcal{X}}}\right|^{-2(q-1)}\left|1 - \frac{\beta_ph(p)}{p^{1/2+\sigma + im/\mathcal{X}}}\right|^{-2(q-1)} . 
\end{split}
\end{align}

 Note that when  $- e^k/\log x \leq \sigma \leq 1/(100\log(2q))$, we have $\mathcal{X} = \min\{ \log x/e^{k+1}, 1/|\sigma| \} = 1/\sigma$, then \eqref{alpha} and \eqref{Eprod} imply that the expectation of the part of the Euler product over primes $e^{1/\sigma} < p \leq x^{e^{-(k+1)}}$ equals to 
 \[ \exp\Big (O \Big(\sum_{e^{1/\sigma} < p \leq x^{e^{-(k+1)}}} \frac{q^2}{p^{1+2\sigma}} + \frac{q^3}{e^{1/(2\sigma)}} \Big) \Big)=e^{O(q^2)}. \]
 Otherwise, if $|\sigma| \leq e^{k+1}/\log x$, then $\mathcal{X} = \log x/e^{k+1}$. Thus in either case,  we deduce from the above discussions and the independence of $h(p)$ for different primes $p$ that for $- e^k/\log x \leq \sigma \leq 1/(100\log(2q))$,
\begin{align*}
\begin{split}  
 &  (|m|+1)^{2(q-1)} \E |F_k(\tfrac12+\sigma)|^{2} |F_{k,f}( \tfrac12+\sigma + \tfrac{im}{\mathcal{X}})|^{2(q-1)} \\
& = e^{O(q^{2})} (|m|+1)^{2(q-1)} \\
& \hspace*{1cm} \times \E \prod_{100q^2 < p \leq e^{\mathcal{X}}} \left|1 - \frac{\alpha_ph(p)}{p^{1/2+\sigma}}\right|^{-2}\left|1 - \frac{\beta_ph(p)}{p^{1/2+\sigma}}\right|^{-2}  \left|1 - \frac{\alpha_ph(p)}{p^{1/2+\sigma + im/\mathcal{X}}}\right|^{-2(q-1)}\left|1 - \frac{\beta_ph(p)}{p^{1/2+\sigma + im/\mathcal{X}}}\right|^{-2(q-1)} .
\end{split}
\end{align*}
Here we observe that our assumptions on $\sigma$, $k$ and $q$ imply that $e^{\mathcal{X}}>100q^2$ so that the product above is not empty. We then apply \eqref{Eprod1}, getting
\begin{align*}
\begin{split}  
 (|m|+1)^{2(q-1)} & \E |F_k(\tfrac12+\sigma)|^{2} |F_{k,f}(\tfrac12+\sigma + \tfrac{im}{\mathcal{X}})|^{2(q-1)} \\
& = e^{O(q^{2})} (|m|+1)^{2(q-1)} \Big(\frac{\mathcal{X}}{\log q} \Big)^{1 + (q-1)^2} \Big(1 + \frac{\mathcal{X}}{(|m|+1)\log q} \Big)^{2(q-1)} = e^{O(q^{2})} \Big(\frac{\mathcal{X}}{\log q} \Big)^{q^2} .
\end{split}
\end{align*}

 We apply the above with \eqref{upperq-1}, \eqref{holderbigq-1} and \eqref{upperq-2} to see that
\begin{align*}
\begin{split} 
\Big\| \int\limits_{-1/2}^{1/2} |F_{k,f}(\tfrac12+\sigma + it)|^2 \dif t \Big\|_{q}  
 \ll \frac{1}{\mathcal{X}} \Big( \mathcal{X} \cdot e^{O(q^2)} \sum_{|m| \leq \mathcal{X}} \frac{1}{(|m|+1)^2} \Big(\frac{\mathcal{X}}{\log q} \Big)^{q^2} \Big)^{1/q} = \frac{1}{\mathcal{X}} \Big( \mathcal{X} \cdot e^{O(q^2)} \Big(\frac{\mathcal{X}}{\log q} \Big)^{q^2} \Big)^{1/q} .
\end{split}
\end{align*}
 One checks that this readily leads to the estimation given in \eqref{Ehalfq} for the case of $h(n)$ being a Steinhaus random multiplicative function. \newline

We next proceed as in the Proof of Key Proposition 5.2 for $q \geq 2$ in \cite{harper2020moments} to see that for $h$ being Rademacher, when $q \geq 2$ and $|N| \geq 1$, we have for any $|n| \leq \mathcal{X}/2 + 1$ and any $|t| \leq 1/(2\mathcal{X})$, 
\begin{align*}
\begin{split} 
\Big( \sum_{|m| \leq \frac{\mathcal{X}}{2} + 1} |F_{k,f}(\tfrac12+\sigma + i(\tfrac{m}{\mathcal{X}} + N + t))|^2 \Big)^{q-1}   \leq  e^{O(q)} \sum_{|m| \leq \mathcal{X}/2 + 1 } \frac{(|m-n|+1)^{2(q-1)} }{(|m-n|+1)^2} |F_{k,f}(\tfrac12+\sigma + i(\tfrac{m}{\mathcal{X}} + N + t))|^{2(q-1)} .
\end{split}
\end{align*}
  In view of \eqref{radupperq-1}, we see that in order to establish \eqref{EhalfqNlargeR}, it remains to estimate terms of the form 
$$ (|m-n|+1)^{2(q-1)} \E |F_{k,f}(\tfrac{1}{2}+\sigma + i(\tfrac{n}{\mathcal{X}} + N + t))|^{2} |F_{k,f}(\tfrac{1}{2}+\sigma + i(\tfrac{m}{\mathcal{X}} + N + t))|^{2(q-1)} . $$
   We now use the definition of $F_{k,f}(s)$ to write it as an Euler product to see that, similar to the Steinhaus case, the contribution from primes $p \leq 100q^2$ to the expectation above is trivially $e^{O(q^{2}/\log q)}$. Also, by \eqref{alpha} and \eqref{EprodR}, the contribution from primes $e^{\mathcal{X}} < p \leq x^{e^{-(k+1)}}$ is $e^{O(q^2)}$. Moreover, observe that the imaginary shifts $\frac{n}{\mathcal{X}} + N + t, \frac{m}{\mathcal{X}} + N + t$ are $\asymp |N| \gg 1$, we then apply \eqref{Eprod1R}, \eqref{Eprod1R1} to treat the contribution from primes $100q^2 \leq p \leq e^{\mathcal{X}}$ to see that
\begin{align*}
\begin{split} 
& (|m-n|+1)^{2(q-1)} \E |F_{k,f}(\tfrac{1}{2}+\sigma + i(\tfrac{n}{\mathcal{X}} + N + t))|^{2} |F_{k,f}(\tfrac{1}{2}+\sigma + i(\tfrac{m}{\mathcal{X}} + N + t))|^{2(q-1)}\\
\ll & e^{O(q^2)} \min\Big\{\frac{\mathcal{X}}{\log q}, |N|^{\frac{1}{100}}\Big\}^{(q-1)^2 + 3(q-1)} (|m-n|+1)^{2(q-1)} \Big(\frac{\mathcal{X}}{\log q} \Big)^{1 + (q-1)^2} \Big(\frac{\mathcal{X}}{(|m-n|+1)\log q} \Big)^{2(q-1)} .
\end{split}
\end{align*}
 We then proceed as in the Steinhaus case to see that the estimation given in \eqref{EhalfqNlargeR} is valid. \newline

It remains to prove \eqref{Ehalfq} for the Rademacher case. For this, we apply H\"{o}lder's inequality as in the Steinhaus case.  For any $|t| \leq 1/(2\mathcal{X})$, 
\begin{align*}
\begin{split} 
& \sum_{|n| \leq \mathcal{X}/2 + 1} |F_{k,f}(\tfrac{1}{2}+\sigma + i(\tfrac{n}{\mathcal{X}} + t))|^2 \Big( \sum_{|m| \leq \mathcal{X}/2 + 1} |F_{k,f}( \tfrac12+\sigma + i(\tfrac{m}{\mathcal{X}} + t))|^2 \Big)^{q-1} \\
 \leq & e^{O(q)} \sum_{|n| \leq \mathcal{X}/2 + 1} |F_{k,f}(\tfrac{1}{2}+\sigma + i(\tfrac{n}{\mathcal{X}} + t))|^2  \cdot \sum_{|m| \leq \mathcal{X}/2 + 1} \frac{(|m|+1)^{2(q-1)}}{(|m|+1)^2} |F_{k,f}(\tfrac12+\sigma + i(\tfrac{m}{\mathcal{X}} + t))|^{2(q-1)} .
\end{split}
\end{align*}
  With the above estimation, we then apply \eqref{Eprod1R}, \eqref{Eprod1R1} again and follow the proof of Key Proposition 5.2 for $q \geq 2$ in \cite{harper2020moments} to see that the desired estimation given in \eqref{Ehalfq} is valid for the Rademacher case. 
\newline
\newline
{\it ii) The case $1 < q < 2$. }
\newline
\newline
 The proof of Rademacher case is similar to the Steinhaus case as shown in the Proof of Key Proposition 5.2 for $1<q<2$ in \cite{harper2020moments}. We will only consider the Steinhaus case here.  We set $C = C(q) = e^{1/(q-1)}$ and $D := \lfloor \frac{1}{q-1} \rfloor$.  For each $d \geq 1$ and $|m| \leq \mathcal{X}$, we set 
$$ G_{d}(m) := \prod_{p \leq e^{\mathcal{X}/C^d}}  \left|1 - \frac{\alpha_ph(p)}{p^{1/2+\sigma+im/\mathcal{X}}}\right|^{-2}\left|1 - \frac{\beta_ph(p)}{p^{1/2+\sigma+im/\mathcal{X}}}\right|^{-2} , \;\;\;\;\; \text{and} \;\;\;\;\; H_{d}(m) := \frac{|F_{k,f}(\frac{1}{2}+\sigma + \frac{im}{\mathcal{X}})|^2}{G_{d}(m)}, $$
 where we define an empty product to be $1$ throughout the paper. We also write for brevity $G_{d} = G_{d}(0)$ and $H_{d} = H_{d}(0)$.

 It is shown in the proof of Key Proposition 5.1 for $1< q < 2$ in \cite{harper2020moments} that we have 
\begin{align*}
\begin{split}  
  \E & |F_{k,f}(\tfrac{1}{2}+\sigma)|^2 \Big( \sum_{|m| \leq \mathcal{X}} |F_{k,f}(\tfrac{1}{2}+\sigma + \tfrac{im}{\mathcal{X}})|^2 \Big)^{q-1} \\
& \ll D^{q-1} \sum_{r \leq \frac{(q-1)\log\mathcal{X}}{D} + 1} \Big( \E \max_{(r-1)D < d \leq rD} G_{d}^{q} H_{d} C^{2(q-1)d} \Big)^{2-q} \Big( \sum_{(r-1)D < d \leq rD} \E \frac{G_{d}^{q-1} H_{d}}{C^{2(2-q)d}} \sum_{C^{d-1} \leq |m| \leq C^{d}} G_{d}(m) H_{d}(m) \Big)^{q-1} \\
& =  D^{q-1} \sum_{r \leq \frac{(q-1)\log\mathcal{X}}{D} + 1} \Big( \E \max_{(r-1)D < d \leq rD} G_{d}^{q} H_{d} C^{2(q-1)d} \Big)^{2-q} \\
& \hspace*{3cm} \times \Big( \sum_{(r-1)D < d \leq rD} \frac{1}{C^{2(2-q)d}} \sum_{C^{d-1} \leq |m| \leq C^{d}} \E G_{d}^{q-1}G_{d}(m) \E H_{d} H_{d}(m) \Big)^{q-1},
\end{split}
\end{align*}
 where the last relation above follows from the independence of $h(p)$ for different primes. Here we adopt the convention that the term $m=0$ is included in $\sum_{C^{d-1} \leq |m| \leq C^{d}}$ when $d=1$, and that any terms with $|m| > \mathcal{X}$ or  $C^{d-1} > \mathcal{X}$ are omitted from all sums since the corresponding sum is empty. \newline

  We now treat $\E G_{d}^{q-1}G_{d}(m) \E H_{d} H_{d}(m)$ by writing $G_{d}^{q-1}G_{d}(m)$ and $H_{d} H_{d}(m)$ into Euler products using their definitions, and then applying Proposition \ref{prop1} to estimate them.  In this process, we note that similar to our discussions on the case $q \geq 2$ above, we see that the products over $p \leq 100q^2$ and $p > e^{\mathcal{X}}$ are uniformly bounded. We then apply \eqref{Eprod1} to deal with the complementary products over $100q^2 <p \leq e^{\mathcal{X}}$. Analogue to what is shown in the proof of Key Proposition 5.2 for $1< q < 2$ in \cite{harper2020moments}, we obtain that
\begin{eqnarray}
 \sum_{(r-1)D < d \leq rD} \frac{1}{C^{2(2-q)d}} \sum_{C^{d-1} \leq |m| \leq C^{d}} \E G_{d}^{q-1}G_{d}(m) \E H_{d} H_{d}(m) \ll  C^{O(1)} \mathcal{X}^{q^2} C^{-(q-1)^{2}(r-1)D} .  \nonumber
\end{eqnarray}

 Next, we estimate $\E \max_{(r-1)D < d \leq rD} G_{d}^{q} H_{d} C^{2(q-1)d}$. To this end, we first use the definition of $G_{d}, H_{d}$.  This leads to
\begin{align*}
\begin{split}   
 \E \max_{(r-1)D < d \leq rD} & G_{d}^{q} H_{d} C^{2(q-1)d} \\
 & = \E |F_{k,f}(\tfrac{1}{2}+\sigma)|^{2} \cdot \tilde{\E} \max_{(r-1)D < d \leq rD} \prod_{p \leq e^{\frac{\mathcal{X}}{C^d}}}  \left|1 - \frac{\alpha_ph(p)}{p^{1/2+\sigma}}\right|^{-2(q-1)}\left|1 - \frac{\beta_ph(p)}{p^{1/2+\sigma}}\right|^{-2(q-1)} C^{2(q-1)d} , 
\end{split}
\end{align*}
where we define $\tilde{\E}$ to be the expectation under the measure $\tilde{\p}(A) = \frac{\E |F_{k,f}(1/2+\sigma)|^{2} \textbf{1}_{A}}{\E |F_{k,f}(1/2+\sigma)|^{2}}$ for each event $A$, and where $\textbf{1}$ denotes the indicator function. Here, as noted in the proof of Key Proposition 5.1 for $1< q < 2$ in \cite{harper2020moments}, the independence of the $h(p)$ means that if an event $A$ does not involve a particular prime, then the expectation of that term will factor out and cancel between the numerator and denominator. Similarly, the random variables $h(p)$ are still independent under the measure $\tilde{\p}$. \newline

  We now apply Proposition \ref{prop1} to evaluate $\E |F_{k,f}(\frac{1}{2}+\sigma)|^{2}$. Similar to our discussions above, apart from a constant factor,  we may restrict the Euler product involved over $100q^2 < p \leq e^{\mathcal{X}}$. We then apply \eqref{Eprod3} to see that $\E |F_{k,f}(\frac{1}{2}+\sigma)|^{2} \asymp \mathcal{X}$. We now set $\lambda_{d} := \tilde{\E} L_d$, where
$$L_{d} := \prod_{p \leq e^{\mathcal{X}/C^d}}  \left|1 - \frac{\alpha_ph(p)}{p^{1/2+\sigma}}\right|^{-2(q-1)/1.01}\left|1 - \frac{\beta_ph(p)}{p^{1/2+\sigma}}\right|^{-2(q-1)/1.01}.$$ 

  By the independence of the $h(p)$ and \eqref{Eprod3},
\begin{align*}
\begin{split}
\lambda_{d} = & \frac{\E \prod_{p \leq e^{\mathcal{X}/C^d}}  \left|1 - \frac{\alpha_p h(p)}{p^{1/2+\sigma}}\right|^{-2(1+(q-1)/1.01)}\left|1 - \frac{\beta_p h(p)}{p^{1/2+\sigma}}\right|^{-2(1+(q-1)/1.01)}}{\E \prod_{p \leq e^{\mathcal{X}/C^d}}  \left|1 - \frac{\alpha_p h(p)}{p^{1/2+\sigma}}\right|^{-2}\left|1 - \frac{\beta_p h(p)}{p^{1/2+\sigma}}\right|^{-2}} \asymp  \Big(1 + \frac{\mathcal{X}}{C^d} \Big)^{2(q-1)/1.01 + (q-1)^2/1.01^2}, \\
\tilde{\E} L_{d}^{1.01} \asymp & \Big(1 + \frac{\mathcal{X}}{C^d} \Big)^{2(q-1) + (q-1)^2}. 
\end{split}
\end{align*}
 We deduce from the above and argue as in the proof of Key Proposition 5.2 for $1< q < 2$ in \cite{harper2020moments} that
\begin{eqnarray}
\E \max_{(r-1)D < d \leq rD} G_{d}^{q} H_{d} C^{2(q-1)d} 
& \ll & \mathcal{X}^{1+2(q-1)} \Big(1 + \frac{\mathcal{X}}{C^{(r-1)D + 1}}\Big)^{(q-1)^2/1.01} \tilde{\E} \max_{(r-1)D < d \leq rD} \Big(\frac{L_d}{\lambda_{d}} \Big)^{1.01}. \nonumber
\end{eqnarray}
 Similar to what is shown in the proof of Key Proposition 5.1 for $1< q < 2$ in \cite{harper2020moments}, we see that the sequence of random variables $\left( L_{rD}/\lambda_{rD} \right), \left( L_{rD-1}/\lambda_{rD-1} \right), ..., \left( L_{(r-1)D+1}/\lambda_{(r-1)D+1} \right)$  form a non-negative submartingale relative to $\tilde{\p}$ and to the sigma algebras generated by $(h(p))_{p \leq e^{\mathcal{X}/C^{rD}}}, (h(p))_{p \leq e^{\mathcal{X}/C^{rD-1}}}, ..., (h(p))_{p \leq e^{\mathcal{X}/C^{(r-1)D+1}}}$. We may thus apply Proposition \ref{prop4} and follow the arguments in the proof of Key Proposition 5.2 for $1< q < 2$ in \cite{harper2020moments} to see that the expression given in \eqref{Ehalfq} is valid for $1 < q < 2$.  This concludes our proof of the proposition. 
\end{proof}

\subsection{Proof of \eqref{mainestimationlower}}
\label{secmainlower}

\subsubsection{Lower bounds, Steinhaus case}
\label{rmfiisteinhauslowermain}

We first apply Proposition \ref{propstlower} to prove \eqref{mainestimationlower} for the Steinhaus case.  Proceeding as in the proof of the lower bound in the Steinhaus case in \cite{harper2020moments} upon recalling that $F_f(s)$ denotes the Euler product of $h(n)$ over $x$-smooth numbers for the Steinhaus case, for $q \geq 1$, we get, writing $R = 1/(2 \log x)$ for convenience,
\begin{align*}
\begin{split}
 \E \Big( \int\limits_{- 1/2}^{1/2}  & |F_f( \tfrac12+ \tfrac{4Vq}{\log x} + it)|^2 \dif t \Big)^q \gg  \log x \cdot \E \Big( \int\limits_{-R}^{R} |F_f(\tfrac12+ \tfrac{4Vq}{\log x} + it)|^2 \dif t \Big)^q  \\
 \gg & \frac 1{\log^{q}x} \E \exp\Big(2q \int\limits_{- R}^{R} \log x \cdot \log|F_f(\tfrac12+\tfrac{4Vq}{\log x} + it)| \dif t \Big ) \\
& = \frac 1{\log^{q}x} \E \prod_{p \leq x} \exp\Big(-2q \int\limits_{- R}^{R} \log x \cdot \Re\log\Big(1 - \frac{\alpha_ph(p)}{p^{1/2+4Vq/\log x + it}}\Big) \dif t \\
& \hspace*{4cm} -2q \int\limits_{-R}^{R} \log x \cdot \Re\log\Big(1 - \frac{\beta_ph(p)}{p^{1/2+4Vq/\log x + it}}\Big) \dif t \Big) .
\end{split}
\end{align*}

 We then apply \eqref{alphalambdarel} and argue as in the proof of the lower bound in the Steinhaus case in \cite{harper2020moments} to see that
\begin{align}
\label{stlowerstep1}
\begin{split}
\E \Big( \int\limits_{- 1/2}^{1/2} |F_f(\tfrac12+ \tfrac{4Vq}{\log x} + it)|^2 \dif t \Big)^q 
\geq  \frac{e^{O(q)}}{\log^{q-1}x} \prod_{p \leq x} \E \exp\Big(2q \Re \Big( \frac{\lambda(p)h(p)}{p^{1/2+4Vq/\log x}} \log x \int\limits_{- R}^{R} e^{-it\log p} dt + \frac{(\alpha^2_p+\beta^2_p)h(p)^2}{2 p^{1+8Vq/\log x}} \Big)\Big).
\end{split}
\end{align}

  We further proceed as in the proof of the lower bound in the Steinhaus case in \cite{harper2020moments} with the observation that $\E (\Re h(p))^2 = 1/2$, arriving at 
\begin{align*}
\begin{split}
 \E \exp\Big(2q & \Re \Big( \frac{\lambda(p)h(p)}{p^{1/2+4Vq/\log x}} \log x \int\limits_{- R}^{R} e^{-it\log p} \dif t + \frac{(\alpha^2_p+\beta^2_p)h(p)^2}{2 p^{1+8Vq/\log x}} \Big)\Big)  \\
= &
\begin{cases}
  \exp\Big( O\Big(\frac{q}{\sqrt{p}}\Big) \Big),  \quad & p < 100q^2, \\
 1 + \frac{\lambda^2(p)q^2}{p^{1+8Vq/\log x}} + O\Big(\frac{q^2 \log p}{p^{1 + 8Vq/\log x} \log x} + \frac{q^3}{p^{3/2}}\Big), \quad & 100q^2 \leq p \leq x. 
\end{cases}
\end{split}
\end{align*}

 We apply the above estimation in \eqref{stlowerstep1} to see that
\begin{align}
\label{Elower}
\begin{split}
\E \Big( \int\limits_{- 1/2}^{1/2} & |F_f( \tfrac12+\tfrac{4Vq}{\log x} + it)|^2 \dif t \Big)^q  \\
& \geq \frac{e^{O(q)}}{\log^{q-1}x} \prod_{p < 100q^2} \exp \Big( O\Big(\frac{q}{\sqrt{p}}\Big) \Big) \prod_{100q^2 \leq p \leq x} \exp \Big(\frac{\lambda^2(p)q^2}{p^{1+8Vq/\log x}} + O\Big(\frac{q^2 \log p}{p^{1 + 8Vq/\log x} \log x} + \frac{q^3}{p^{3/2}} \Big) \Big)  \\
 & = \frac{e^{O(q^{2}/\log(2q))}}{\log^{q-1}x} \prod_{100q^2 \leq p \leq x} \exp \Big(\frac{\lambda^2(p)q^2}{p^{1+8Vq/\log x}} \Big). 
\end{split}
\end{align}

Recasting the last product as
\begin{align}
\label{Eprodpmiddlrange}
\begin{split}
& \prod_{100q^2 \leq p \leq x} \exp \Big(\frac{\lambda^2(p)q^2}{p^{1+8Vq/\log x}} \Big)=\exp \Big (q^2\sum_{p} \frac{\lambda^2(p)}{p^{1+8Vq/\log x}} -q^2\sum_{ p \leq 100q^2} \frac{\lambda^2(p)}{p^{1+8Vq/\log x}} + q^2\sum_{p > x} \frac{\lambda^2(p)}{p^{1+8Vq/\log x}} \Big).
\end{split}
\end{align}
 
By \eqref{lambdabound}, \eqref{merten} and partial summation, we get
\begin{align}
\label{Eprodplarge}
\begin{split}
& -q^2\sum_{ p >x} \frac{\lambda^2(p)}{p^{1+8Vq/\log x}}=O\Big(q^2\sum_{ p >x} \frac{1}{p^{1+8Vq/\log x}}\Big)=O(1).
\end{split}
\end{align}

  Observe also that by \eqref{merten} and \eqref{merten1},
\begin{align*}
\begin{split}
& -q^2\sum_{ p \leq 100q^2} \frac{\lambda^2(p)}{p^{1+8Vq/\log x}}=-q^2\sum_{ p \leq 100q^2} \frac{\lambda^2(p)}{p}+O\Big(\sum_{ p \leq 100q^2} \frac{q^2\log p}{p\log x}\Big)=-q^2 \log \log q+O(q^2).
\end{split}
\end{align*}
 It follows that
\begin{align}
\label{Eprodsmall}
\begin{split}
& \exp\left (-q^2\sum_{ p \leq 100q^2} \frac{\lambda^2(p)}{p^{1+8Vq/\log x}} \right )=\frac{e^{O(q^2)}}{\log^{q^2}(2q)}.
\end{split}
\end{align}

  We next deduce from \eqref{Fsdef} and \eqref{Fsdecomp} that
\begin{align}
\label{Eprodall}
\begin{split}
 \exp\left (q^2\sum_{p} \frac{\lambda^2(p)}{p^{1+8Vq/\log x}} \right )=& e^{O(q^2)}F^{q^2}(1+8Vq/\log x) \\
 & =e^{O(q^2)}\frac {\zeta^{q^2}(1+8Vq/\log x) L^{q^2}(1+8Vq/\log x, \operatorname{sym}^2 f)}{\zeta^{q^2}(2(1+8Vq/\log x))} = e^{O(q^2)} \Big(\frac{\log x}{Vq}\Big)^{q^2}.
\end{split}
\end{align}
 where the last equality above follows from \eqref{Lsymbound} and \eqref{zeta2sbound1}.  Applying \eqref{Eprodpmiddlrange}--\eqref{Eprodall} in \eqref{Elower}, we obtain
\begin{align*}
\begin{split}
\E \Big( \int\limits_{- 1/2}^{1/2} |F_f( \tfrac12+\tfrac{4Vq}{\log x} + it)|^2 \dif t \Big)^q  
 \geq   \frac{e^{O(q^{2})}}{\log^{q-1}x} \Big(\frac{\log x}{Vq \log(2q)}\Big)^{q^2} . 
\end{split}
\end{align*}

  Substituting the above into \eqref{hlambdalower1} allows us to obtain that
$$ \Big\| \sum_{n \leq x} h(n)\lambda(n) \Big\|_{2q} \gg \sqrt{\frac{x}{\log x}} \Big(\frac{e^{O(q)}}{\log^{(q-1)/2q}x} \Big(\frac{\log x}{Vq \log(2q)} \Big)^{q/2}  - \frac{C}{e^{Vq/2}} \Big\| \int\limits_{-1/2}^{1/2} |F_f(1/2 + \frac{2Vq}{\log x} + it)|^2 \dif t \Big\|_{q}^{1/2} \Big) . $$

  We then apply \eqref{Ehalfq} with $k=-1$ to deduce that, when $2Vq/\log x \leq 1/(100\log(2q))$, we have
$$ \Big\| \sum_{n \leq x} h(n)\lambda(n) \Big\|_{2q} \gg \sqrt{\frac{x}{\log x}} \Biggl(\frac{e^{O(q)}}{\log^{(q-1)/2q}x} \left(\frac{\log x}{Vq \log(2q)}\right)^{q/2}  - \frac{C e^{O(q)}}{e^{Vq/2}} \frac{(Vq)^{(q-1)/2q}}{\log^{(q-1)/2q}x} \left(\frac{\log x}{Vq \log(2q)}\right)^{q/2} \Biggr) . $$
  We then set $V$ to be a sufficiently large fixed constant so that, the subtracted term will be negligible compared with the first term, and our Theorem 1 lower bound will be proved. It only remains to note that the condition $2Vq/\log x \leq 1/(100\log(2q))$ is then satisfied provided $q \leq c\log x/\log\log x$, for a sufficiently small fixed constant $c > 0$. 

\subsubsection{Lower bounds for the Rademacher case}

   We now apply \eqref{hlambdalower2} to prove \eqref{mainestimationlower} for the Rademacher case. For this, we proceed as in the proof of the lower bound in the Rademacher case in \cite{harper2020moments} upon recalling that $F_f(s)$ denotes the Euler product of $h(n)$ over $x$-smooth numbers for the Rademacher case to see that for $q \geq 1$ we have, writing $R=1/(2 \log x)$ as before, $R_1 = (k-1/2)/\log x$ and $R_2 = (k+1/2)/\log x$
\begin{align}
\label{radlowerstep1}
\begin{split}
& \E\Big( \int\limits_{- 1/2}^{1/2} |F_f(1/2+ \frac{4Vq}{\log x} + it)|^2 \dif t \Big)^q \gg \frac 1{\log^{q}x} \E \exp\Big(2q \int\limits_{- R}^{R} \log x \cdot \log|F_f(1/2+\frac{4Vq}{\log x} + it)| \dif t \Big )   \\
 = & \frac{1}{\log^{q}x} \sum_{|k| \leq \frac{\log x - 1}{2}} \prod_{p \leq x} \E \exp\Big( 2q \int\limits_{R_1}^{R_2} \log x \cdot \Re \log\Big( 1 + \frac{\lambda(p)h(p)}{p^{1/2+4Vq/\log x + it}} \Big) \dif t \Big)  \\
 = & \frac{1}{\log^{q}x} \sum_{|k| \leq \frac{\log x - 1}{2}} \prod_{p \leq x} \E \exp\Big(2q \int\limits_{R_1}^{R_2} \log x \cdot \Big( \frac{\lambda(p)h(p) \cos(t\log p)}{p^{1/2+4Vq/\log x}} - \frac{\lambda^2(p)\cos(2t\log p)}{2p^{1+8Vq/\log x}} + O\Big(\frac{1}{p^{3/2}} \Big) \Big) \dif t \Big) \\
\geq & \frac{e^{O(q)}}{\log^{q}x} \sum_{|k| \leq \frac{\log x - 1}{2}} \prod_{p \leq x} \E \exp\Big(2q \Big( \frac{\lambda(p)h(p)}{p^{1/2+4Vq/\log x}} \log x \int\limits_{R_1}^{R_2} \cos(t\log p) \dif t - \frac{\lambda^2(p)\cos(\frac{2k\log p}{\log x})}{2 p^{1+8Vq/\log x}} \Big)\Big).
\end{split}
\end{align}

Next, when $100q^2 \leq p \leq x$ we have, mindful the Taylor expansion of the exponential (and the fact that $h(p)^2 \equiv 1$), that
\begin{equation*}
\begin{split}
 \E \exp\Big( 2q & \Big( \frac{\lambda(p)h(p)}{p^{1/2+4Vq/\log x}} \log x \int\limits_{R_1}^{R_2} \cos(t\log p) \dif t - \frac{\lambda^2(p)\cos(\frac{2k\log p}{\log x})}{2 p^{1+8Vq/\log x}} \Big)\Big) \\
& =  \E \Big( 1 + 2q \Big( \frac{\lambda(p)h(p)}{p^{1/2+4Vq/\log x}} \log x \int\limits_{R_1}^{R_2} \cos(t\log p) \dif t - \frac{\lambda^2(p)\cos(\frac{2k\log p}{\log x})}{2 p^{1+8Vq/\log x}} \Big) + \frac{2q^2 \lambda^2(p)\cos^2(\frac{k\log p}{\log x})}{p^{1 + 8Vq/\log x}}  \\
& \hspace*{2cm} + O\Big(\frac{q^2 \log p}{p^{1 + 8Vq/\log x} \log x} + \frac{q^3}{p^{3/2}} \Big) \Big) . 
\end{split}
\end{equation*}
Using the cosine identity $\cos^2 \theta = \tfrac12(1 + \cos( 2 \theta))$, and the fact that $\E h(p) = 0$, we find the above equals
$$ = 1 + \frac{\lambda^2(p)(q^2 + (q^2 - q)\cos(\frac{2k\log p}{\log x}))}{p^{1+8Vq/\log x}} + O\Big(\frac{q^2 \log p}{p^{1 + 8Vq/\log x} \log x} + \frac{q^3}{p^{3/2}}\Big) . $$
When $p < 100q^2$, we shall instead use the trivial bound $\exp( O(q/\sqrt{p}) )$. Inserting these into \eqref{radlowerstep1}, we get
\begin{align}
\label{ElowerR}
\begin{split}  \E\left( \int\limits_{- 1/2}^{1/2} |F_f(\tfrac12+\tfrac{4Vq}{\log x} + it)|^2 \dif t \right)^q \geq \frac{e^{O(\frac{q^{2}}{\log(2q)})}}{\log^{q}x} \sum_{|k| \leq \frac{\log x - 1}{2}} \prod_{100q^2 \leq p \leq x} \exp\Big( \frac{\lambda^2(p)(q^2 + (q^2 - q)\cos(\frac{2k\log p}{\log x}))}{p^{1+8Vq/\log x}} \Big ) . \end{split}
\end{align}

For $2 \leq q \leq c\log x/\log\log x$, we keep only the $k=0$ term on the right-hand side above and then argue as in \eqref{Eprodall} to get 
\begin{align*}
\begin{split} 
\E \Big( \int\limits_{- 1/2}^{1/2} |F_f( \tfrac12+\tfrac{4Vq}{\log x} + it)|^2 \dif t \Big)^q \geq \frac{e^{O(\frac{q^{2}}{\log(2q)})}}{\log^{q}x} \prod_{100q^2 \leq p \leq x} \exp \Big( \frac{\lambda^2(p)(2q^2 - q)}{p^{1+8Vq/\log x}} \Big ) = \frac{e^{O(q^{2})}}{\log^{q}x} \Big(\frac{\log x}{Vq \log(2q)}\Big)^{2q^2 - q}.
\end{split}
\end{align*}
 
  Inserting this into \eqref{hlambdalower2}, and applying \eqref{Ehalfq} with $k=-1$ and $\sigma = 2Vq/\log x$ to control the subtracted term there, we find that
$$ \Big\| \sum_{n \leq x} h(n)\lambda(n) \Big\|_{2q} \gg \sqrt{\frac{x}{\log x}} \Big(\frac{e^{O(q)}}{\log^{1/2}x} \Big(\frac{\log x}{Vq \log(2q)}\big)^{q - 1/2}  - \frac{C e^{O(q)}}{e^{Vq/2}} \frac{(Vq)^{1/2}}{\log^{1/2}x} \Big(\frac{\log x}{Vq \log(2q)}\Big)^{q - 1/2} \Big) . $$
If $V$ is a sufficiently large constant, the subtracted term is negligible compared with the first term and this leads to the validity of \eqref{mainestimationlower} for $2 \leq q \leq c\log x/\log\log x$. \newline

When $1 \leq q <2$, we first apply \eqref{merten1} and partial summation to see that
\begin{eqnarray*}
\sum_{100q^2 \leq p \leq x} \frac{\lambda^2(p)\cos(\frac{2k\log p}{\log x})}{p^{1+8Vq/\log x}} & = & \sum_{2 \leq p \leq x^{1/V}} \frac{\lambda^2(p)\cos(\frac{2k\log p}{\log x})}{p} + O(1) = \int\limits_{\log 2}^{\log(x^{1/V})} \frac{\cos(\frac{2k}{\log x} u)}{u} \dif u + O(1) \\
& = & \log\min\{ \log(x^{1/V}) , \log(x^{1/(1+|k|)}) \} + O(1) .
\end{eqnarray*}
  
  We apply the above estimation in \eqref{ElowerR} and argue as the case $2 \leq q \leq \frac{c\log x}{\log\log x}$ to see that
\begin{align*}
\begin{split} 
\E\Big( \int\limits_{- 1/2}^{1/2} |F_f(\tfrac12+\tfrac{4Vq}{\log x} + it)|^2 \dif t \Big)^q \gg & \frac{1}{\log^{q}x} \Big(\frac{\log x}{V}\Big)^{q^2} \sum_{|k| \leq \frac{\log x - 1}{2}} \min \Big\{ \frac{\log x}{V}, \frac{\log x}{1+|k|} \Big\}^{q^2 - q} \\
\gg & \frac{1}{\log^{q}x} \Big(\frac{\log x}{V} \Big)^{q^2 + \max\{1,q^{2}-q\}} \min\Big\{\log\log x, \frac{1}{|q-q_0|} \Big\}.
\end{split}
\end{align*}
Again, inserting this in \eqref{hlambdalower2} leads to the validity of \eqref{mainestimationlower} for $1 \leq q <2$.

\vspace*{.5cm}

\noindent{\bf Acknowledgments.} P. G. is supported in part by NSFC grant 12471003 and L. Z. by the FRG Grant PS71536 at the University of New South Wales.

\bibliography{biblio}
\bibliographystyle{amsxport}

\end{document}